\documentclass[sn-mathphys-num]{sn-jnl}%
\usepackage{epstopdf}
\usepackage{graphicx}%
\usepackage{multirow}%
\usepackage{amsmath,amssymb,amsfonts}%
\usepackage{amsthm}%
\usepackage{mathrsfs}%
\usepackage[title]{appendix}%
\usepackage{xcolor}%
\usepackage{array}
\usepackage{textcomp}%
\usepackage{manyfoot}%
\usepackage{booktabs}%
\usepackage{algorithm}%
\usepackage{algorithmicx}%
\usepackage{algpseudocode}%
\usepackage{listings}%
\usepackage{arydshln}
\theoremstyle{thmstyleone}%
\newtheorem{theorem}{Theorem}%

\newtheorem{proposition}[theorem]{Proposition}%

\newtheorem{lemma}{Lemma}[section]

\theoremstyle{thmstyletwo}%

\theoremstyle{thmstylethree}%
\newtheorem{definition}{Definition}%

\begin{document}

\title[Article Title ]{Modified Block Newton Algorithm for $\ell_0$- Regularized Optimization}

\author[1]{\fnm{} \sur{Yuge Ye}}\email{fujianyyg@163.com}

\author*[1]{\fnm{} \sur{Qingna Li}}\email{qnl@bit.edu.cn}

\affil[1]{\orgdiv{Department of Mathematics and Statistics}, \orgname{Beijing Institute of Technology}, \orgaddress{\street{No.5 Yard, ZhongGuanCun South Street}, \city{Beijing}, \postcode{100081}, \state{Beijing}, \country{China}} }

\abstract{In this paper, we propose a globally convergent Newton type method to solve $\ell_0$ regularized sparse optimization problem. In fact, a line search strategy is applied to the Newton method to obtain global convergence. The Jacobian matrix of the original problem is a block upper triangular matrix. To reduce the computational burden, our method only requires the calculation of the block diagonal. We also introduced regularization to overcome matrix singularity. Although we only use the block-diagonal part of the Jacobian matrix, our algorithm still maintains global convergence and achieves a local quadratic convergence rate. Numerical results demonstrate the efficiency of our method. }

\keywords{sparse optimization, global convergence, Newton method, proximal point mapping
}

\maketitle

\section{Introduction}\label{sec1}
\numberwithin{equation}{section}
In recent years, sparse optimization has become a hot topic of research due to its applications in compressed sensing \cite{candes2006robust, 1542412, donoho2006compressed}, machine learning \cite{wright2010sparse, yuan2012visual}, neural networks \cite{bian2012smoothing, dinh2020sparsity, lin2019toward}, signal and image processing \cite{bian2015linearly, chen2012non, elad2010sparse, elad2010role}, mtrix completion \cite{10.1145/2184319.2184343}. Among these, models with the $\ell_0 $ norm as a penalty term represent a classic approach, which can be expressed as the following $l_0$ regularized optimization,
\begin{align}\label{L0}
\min\limits_{x\in \mathbb{R}^n} f(x) + \lambda \|x\|_0,
\end{align}
where $f: \mathbb{R}^n \to \mathbb{R}$ is twice continuously differentiable and bounded from below, $\lambda > 0$ is the penalty parameter and $\|x\|_0$ is the $\ell_0$ norm of $x$, counting the number of non-zero elements of $x$.

Various methods have been proposed for problem \eqref{L0}, such as the iterative hard-thresholding algorithm (IHT, \cite{blumensath2008iterative}), the forward-backward splitting method (FBS, \cite{attouch2013convergence}), mixed integer optimization method (MIO, \cite{bertsimas2016best}), active set Barzilar
Borwein method (ABB, \cite{cheng2020active}), smoothing proximal gradient method (SPG, \cite{bian2020smoothing}).

We note that these methods are known as the first-order methods. Besides, several
second-order methods have been proposed, including primal dual active set (PDAS, \cite{Ito_2014}), the primal dual active set with continuation (PDASC, \cite{JIAO2015400}),  support detection and rootfinding (SDAR, \cite{huang2018constructive}).
Recently, Zhou et al. have introduced some promising second-order methods for $\ell_0$-norm sparse optimization, such as Newton hard-thresholding  pursuit (NHTP, \cite{zhou2021global}), Smoothing Newtons method for $\ell_{0/1}$ loss optimization (NM01, \cite{zhou2021quadratic}) and subspace Newton method for $\ell_0$-regularized optimization problem (NL0R, \cite{zhou2021newton}). Notably, NL0R is the first method that possesses both of global convergence and local quadratic convergence properties which aiming to solve problem \eqref{L0}. For Newton-type methods, the computation of the Jacobian matrix have always been a significant challenge. Inspired by NL0R, we are motivated to investigate a question which is whether we can develop a Newton-type method for problem \eqref{L0} that without computing the full Jacobian matrix? This motivate the work in this paper.

In this paper, we will propose a modified block Newton's method (MBNL0R) to solve problem \eqref{L0}. Notably, in NL0R, a block upper triangular Jacobian matrix needs to be computed at each iteration. To reduce the computational burden, we only compute the block diagonal Jacobian matrix. Due to the special structure of NL0R, our proposed approach retains the convergence properties inherent to NL0R. Additionally, we also introduced a regularization to overcome matrix singularity. In the convergence analysis, we prove the global convergence and local quadratic convergence rate for our algorithm. Numerical experiments demonstrate the effectiveness of the proposed algorithm.

The organization of the paper is as follows. In section 2, we introduce some preliminaries for problem \eqref{L0}. In section 3, we propose the so-called modified block Newton method (MBNL0R) for $\ell_0$-regularized problem. In section 4, we analyze the global convergence property and the local quadratic convergence rate. We conduct various numerical experiments in section 5 to verify the efficiency of the proposed method. Final conclusions are given in section 6.

Notations. For $x\in \mathbb{R}^n$, $|x| := ( |x_1|, |x_2|, \cdots, |x_n|)^{\top}$ denotes the absolute value of each component of $x $ and ${\rm{supp}} (x)$ be its support set. Let $\mathbb{N}_{n} := [1,2,\cdots, n]$ and $\| \cdot \|$ be the $\ell_2$ norm. We also use $[m]$ to denote $\mathbb{N}_{m}$. Given a set $T \subseteq \mathbb{N}_{n}$, we denote $|T|$ as its cardinality set and $\overline{T}$ as its complementary set. Given a marix $H \in \mathbb{R}^{m\times n}$, let $H_{T,J}$ represent its sub-matrix containing rows indexed by $T \subseteq \mathbb{N}_{m}$ and columns indexed by $J \subseteq \mathbb{N}_{n}$. Let $I \in \mathbb{R}^{n\times n}$ represent the identity matrix.  In particular, we define the sub-gradient and sub-Hessian by
\begin{align*}
&\nabla_{T} f(x) := [\nabla f(x)]_{T},  \nabla^2_{T} f(x) := [\nabla^2 f(x)]_{T,T},\\
&\nabla^2_{T,J} f(x) := [\nabla^2 f(x)]_{T,J},\ \nabla^2_{T :} f(x) := [\nabla^2 f(x)]_{T,\mathbb{N}_n} .
\end{align*}

\section{Preliminaries}
In this section, we review some preliminaries of problem \eqref{L0}. To facilitate a clearer understanding of our proposed method, we also revisit the NL0R method in \cite{zhou2021newton} for solving problem \eqref{L0}.

The definition of $\tau$-stationary point$ (\tau > 0)$ of \eqref{L0} is given as follows \cite{zhou2021newton}.
\begin{definition}
We say that $x$ is a $\tau$-stationary point $(\tau > 0)$ if the following holds:
\begin{align}\label{eq-2-1}
x &\in {\rm{Prox}}_{\tau \lambda\|\cdot \|_0}(x - \tau \nabla f(x)) ,
\end{align}
where ${\rm{Prox}}_{\tau \lambda\|\cdot \|_0}(z)$ is defined by
\begin{align}\label{eq-2-2}
[{\rm{Prox}}_{\tau \lambda\|\cdot \|_0}(z)]_i =
\begin{cases}
z_i,  & \mbox{if }|z_i| > \sqrt{2\tau \lambda}, \\
z_i\ \mbox{or}\  0, & \mbox{if }|z_i| = \sqrt{2\tau \lambda},\\
0, & \mbox{if }|z_i| < \sqrt{2\tau \lambda}.
\end{cases}
\end{align}
\end{definition}
The following proposition provides a necessary and sufficient condition for $\tau$-stationary point.
\begin{proposition} {\rm{\cite[Lemma 1]{zhou2021newton}}}
For problem \eqref{L0}, $x$ is a $\tau$-stationary point, if and only if
\begin{align}\label{eq-2-3}
\begin{cases}
\nabla_i f(x) = 0\  and\  |x_i| > \sqrt{2\tau \lambda},\  & i\in \text{supp} (x),\\
\nabla_i f(x) \le \sqrt{2\tau \lambda},\  & i\not\in \text{supp} (x).\\
\end{cases}
\end{align}
\end{proposition}
Here $\text{supp}(x)$ denotes the support set of $x$, defined by $$\text{supp}(x) := \{ j \in \mathbb{N}_n\ | \ x_j \neq 0\}.$$
Similar to NL0R we also needs the strong smoothness and convexity of $f$.
\begin{definition}
$f$ is strongly smooth about constant $L>0$ if
\begin{align}\label{eq-2-4}
f(z) \le f(x) + \langle \nabla f(x),z-x\rangle + (L/2)\|z-x\|^2,\ \forall x,z\in \mathbb{R}^n ,
\end{align}
$f$ is strongly convex about constant $\mu > 0$ if
\begin{align}\label{eq-2-5}
f(z) \ge f(x) + \langle \nabla f(x),z-x\rangle + (\mu/2)\|z-x\|^2,\ \forall x,z\in \mathbb{R}^n .
\end{align}
\end{definition}
The following results reveals the relation of the minimizer of problem \eqref{L0} and its $\tau$-stationary point.
\begin{proposition} {\rm{\cite[Theorem 1]{zhou2021newton}}} For problem \eqref{L0}, the following results hold.
\begin{enumerate}
    \item[(1)] \textbf{(Necessity)} A global minimizer $x^*$ is also a $\tau$-stationary point for any $0 < \tau < \frac{1}{L}$ if $f$ is strongly smooth with $L > 0$. Moreover,
    \begin{align*}
    x^* = {\rm{Prox}}_{\tau \lambda} \left( x^* - \tau \nabla f(x^*) \right).
    \end{align*}
    \item[(2)] \textbf{(Sufficiency)} A $\tau$-stationary point with $\tau > 0$ is a local minimizer if $f$ is convex. Furthermore, a $\tau$-stationary point with $\tau \geq 1/\ell$ is also a (unique) global minimizer if $f$ is strongly convex with $\ell > 0$.
\end{enumerate}
\end{proposition}
To express the solution of \eqref{eq-2-1} more explicitly, we define
\begin{align}\label{eq-2-6}
T:=T_{\tau}(x;\lambda):=\{ i\in \mathbb{N}_n \mid |x_i - \tau \nabla_{i} f(x)| \ge \sqrt{2\tau \lambda} \} .
\end{align}
By \cite[Theorem 2]{zhou2021newton}, the following holds
\begin{align*}
 x = {\rm{Prox}}_{\tau \lambda\|\cdot \|_0}(z) \Rightarrow
 F_{\tau}(x;T) = 0 \Rightarrow x \in {\rm{Prox}}_{\tau \lambda\|\cdot \|_0}(z).
\end{align*}
where $F_{\tau}(x;T) $ is defined by
\begin{align}\label{eq-2-7}
F_{\tau}(x;T) :=
 \begin{bmatrix}
 \nabla_{T}f(x) \\
 x_{\overline{T}}
 \end{bmatrix}
 = 0.
\end{align}
The idea of NL0R to solve \eqref{L0} is to apply Newton's method to solve system \eqref{eq-2-7}. In each iteration $k$, we define
\begin{align*}
g^k := \nabla f(x^k),\ T_k := \{ i\in \mathbb{N}_n \ | \ |x^k_i - g^k_i | \ge \sqrt{2\tau \lambda} \}.
\end{align*}
The direction vector updated by Newton's method in NL0R is taken as follows
\begin{align}\label{direction2}
 JF_{\tau}(x^k;T_k)d^k = - F_{\tau}(x^k;T_k),
\end{align}
where $JF_{\tau}(x^k;T_k)$ is the Jacobian of $F_{\tau}(x^k;T_k)$ at $x^k$. Note that, \eqref{direction2} is taken as the following equivalent form
\begin{align}\label{eq-zhounewtondk}
\nabla_{T_k}^{2}f(x^k) d_{T_k}^k       &= \nabla_{T_k,\overline{T}_k}^{2}f(x^k) x^k_{\overline{T}_k} - \nabla_{T_k} f(x^k) \nonumber\\
d^k_{\overline{T}_k} &= -x^k_{\overline{T}_k}
\end{align}
If $d^k$  fails to satisfy some descent conditions, NL0R adopts the gradient descent method to update the iteration direction. To achieve global convergence, a modified Armijo-type line search is used for $d^k$.

Note that $\overline{T}_k$ is in general large and the calculation of $\nabla^2_{T_k, \overline{T}_k} f(x^k)$ may takes extensive computational cost. Therefore, a natural question is whether we can avoid computing $\nabla^2_{T_k, \overline{T}_k} f(x^k)$ while keeping the good convergence property. This motivates the work in this paper, which is detailed in Section 3.

\section{Modified Block Newton Method}
In this part, we propose the so-called modified block Newton's method (MBNL0R) for solving problem \eqref{L0}. The idea is as follows.

Note that the Jacobian of $F_{\tau}(x; T)$ is
\begin{align*}
\begin{bmatrix}
\nabla_{T_k}^{2}f(x^k)        & \nabla_{T_k,\overline{T}_k}^{2}f(x^k)\\
0_{\overline{T}_k,T_k} & I_{\overline{T}_k,\overline{T}_k}
\end{bmatrix}
\end{align*}
The eigenvalues of the block matrix on the main diagonal is the eigenvalue of the entire matrix, so we considered to ignore the block matrix in the upper right corner to reduce computational complexity. To overcome the singularity of $\nabla_{T_k}^{2}f(x^k)$, we introduced a regularization term $\mu_k >0$.  The approximate Jacobian matrix is
\begin{align*}
\begin{bmatrix}
\nabla_{T_k}^{2}f(x^k) + \mu_k I_{T_k,T_k}        & 0_{T_k,\overline{T}_k}\\
0_{\overline{T}_k,T_k} & I_{\overline{T}_k,\overline{T}_k}
\end{bmatrix}.
\end{align*}
For brevity, we omit the subscript of the identity matrix $I$. Then the update of direction $d^k$ is taken as follows
\begin{align}\label{direction}
\begin{cases}
&(\nabla_{T_k}^{2}f(x^k) + \mu_k I)d_{T_k}^k = - \nabla_{T_k}f(x^k) , \\
&d_{\overline{T}_k}^k = -x_{\overline{T}_k}^k.
\end{cases}
\end{align}
This means that the support set of $x^{k+1}$ will be located within $T_k$. Namely,
\begin{align}\label{support}
{\rm{supp}} (x^{k+1}) \subset T_k.
\end{align}
The modified Armijo-type line search is taken as follows
\begin{align}\label{2.4}
x^k(\alpha):= \begin{bmatrix}
x^k_{T_k} + \alpha d^k_{T_k}\\
x^k_{\overline{T}_k} + d^k_{\overline{T}_k}
\end{bmatrix}
=\begin{bmatrix}
x^k_{T_k} + \alpha d^k_{T_k}\\
0
\end{bmatrix}.
\end{align}
Interestingly, we observe that the NL0R method eventually fixes the index set $T_k$ within a neighborhood of the solution while maintaining $x_{\overline{T}_k}^k=0$. Thus, when $k$ is sufficiently large, direction updating \eqref{eq-zhounewtondk} and \eqref{direction} become equivalent if we omit the additional regularization term $\mu_k I$ in \eqref{direction}. This observation validates the soundness of our proposed approach. Furthermore, with an appropriate choice of $\mu_k \to 0$, we can also establish the local super-linear or quadratic convergence rate. These aspects will be elaborated in detail in subsequent sections.

Similar to NL0R, to avoid $0$ becoming a $\tau$-stationary point, we also denote the following bound for parameter $\lambda$.
\begin{align*}
\underline{\lambda} := \min\limits_{i} \left\{ \frac{\tau}{2}|\nabla_i f(0)|^2\ :\ \nabla_i f(0) \neq 0  \right\}, \overline{\lambda} := \max\limits_{i} \frac{\tau}{2}|\nabla_i f(0)|^2 .
\end{align*}
The framework of the proposed method is given as following.
\begin{algorithm}[H]
\caption{Modified Block Newton Algorithm for $\ell_0$-Regularized Optimization (MBNL0R)}\label{algorithm-NL0R}
\begin{algorithmic}[1]
\Statex
\textbf{S0.} Initial setting:  $\tau > 0, \delta > 0, \lambda \in (0,\underline{\lambda}), \sigma \in (0,1/2), \beta \in (0,1), x^0, T_{-1} = \emptyset, k = 0, \mu_0 > 0, \epsilon > 0$, $D > 0$.

\Statex
\textbf{S1.} If $\|F_{\tau}(x^k; T_k) \| < \epsilon$, stop. Otherwise if $ S_{k+1} \neq \emptyset$ then $T_{k+1} = \tilde{T}_{k}$, otherwise $ T_{k+1} = T_{k-1}$,
$$ \tilde{T}_{k} = \{ i\in \mathbb{N}_{n} : |x_{i}^{k} - \tau g_{i}^{k}| {\ge} \sqrt{2\tau \lambda} \} .$$

\Statex
\textbf{S2.} Set $\mu_k = \min (\| F_{\tau}(x^k, T_k) \|^2, D)$ and solve \eqref{direction}. If $d^{k}$ satisfies
\begin{align}\label{alg-2-1}
 \langle g^k_{T_{k}}, d^k_{T_{k}} \rangle \le  -\delta \|d^k \|^2 + (1/4\tau)\| x^{k}_{\overline{T}_{k}} \|^2 - \mu_k \|d^k_{T_k}\|^2,
\end{align}
go to \textbf{S3}. Otherwise,
\begin{align}\label{alg-2-3}
d^k_{T_{k}} = -g^k_{T_{k}}, d^{k}_{\overline{T}_{k}} = -x^{k}_{\overline{T}_{k}}.
\end{align}

\Statex
\textbf{S3.} Find the smallest non-negative integer $m_k$ such that
$$ f(x^{k}(\beta^{m_k})) \le f(x^k) + \sigma \beta^{m_k}\langle g^k, d^k \rangle. $$

\Statex
\textbf{S4.} Set $\alpha_k = \beta^{m_k}, x^{k+1} = x^{k}(\alpha_k) = x^k + \alpha_k d^k$, $k = k+1$. Go to \textbf{S1}.
\end{algorithmic}
\end{algorithm}

\section{Global convergence}
Next, we analyze the global convergence of the algorithm.
We define the following parameters
\begin{align}\label{2-1-1}
&\overline{\alpha} := \min \left\{\frac{1-2\sigma}{L/\delta-\sigma}, \frac{2(1-\sigma)\delta}{L},1 \right\} ,\nonumber \\
&\overline{\tau} := \min \left\{\frac{2\overline{\alpha}\delta \beta}{nL^2}, \frac{\overline{\alpha} \beta}{n},\frac{1}{4L}, \frac{\nu}{n(2L+D)} \right\} \nonumber ,\\
&\rho := \min \left\{\frac{2\delta - n\tau L^2}{2}, \frac{2-n\tau}{2} \right\},
\end{align}
where, $\sigma$, $\delta$, $\beta$, $D$ are defined in Algorithm \ref{algorithm-NL0R}, $0< \nu \le \frac{\overline{\alpha} \beta(1-\sigma)}{1-\overline{\alpha} \beta\sigma}$. For convenience, let
\begin{center}
$J_k := T_{k-1}\backslash T_{k}$, $S_{k}:= \tilde{T}_{k}\backslash T_{k-1}$.
\end{center}
The direction $d^k$ can express as
\begin{align}\label{2-1}
\begin{cases}
-d_{\overline{T}_k}^k = x_{\overline{T}_k}^k =
\begin{bmatrix}
x_{T_{k-1}\cap \overline{T}_{k}}^k\\
0
\end{bmatrix}
=
\begin{bmatrix}
x_{T_{k-1}\backslash T_{k}}^k\\
0
\end{bmatrix}
=
\begin{bmatrix}
x_{J_k}^k\\
0
\end{bmatrix} ,\\
Q_{k}d_{T_k}^{k} = -g_{T_k}^{k}.
\end{cases}
\end{align}
where $Q_k = \nabla_{T_k}^{2}f(x^k) + \mu_k I$.
\begin{lemma}\label{lem-2-1}
{\rm{(Descent property)}} Let $f$ be strongly smooth with $L>0$, $\overline{\tau},\rho$ be defined by \eqref{2-1-1}. Then, for any $\tau  \in(0,\overline{\tau})$, we have $\rho >0$ and
\begin{align}\label{lem-2-1-0}
\langle g^{k},d^{k}\rangle \le -\rho\|d^{k}\|^{2} - \frac{\tau}{2}\|g_{T_{k-1}}^k \|^{2}.
\end{align}
\end{lemma}
\begin{proof}
It follows from \eqref{2-1-1} that $\overline{\alpha} \le 1$ and thus $\overline{\alpha} \beta < 1$ due to $\beta \in (0,1)$ . Hence $\overline{\tau} \le \min \{2\delta/(nL^2),2/n\}$, immediately showing that $\rho > 0$ if $\tau \in (0,\overline{\tau})$ . In
 addition, if $d^k$ is updated by \eqref{direction}, then
\begin{align}\label{lem-2-1-1}
\|g_{T_k}^k\| = \|(\nabla^2_{T_k} f(x^k) + \mu_k I) d_{T_k}^k\| \le (L+\mu_k)\|d_{T_k}^k\| ,
\end{align}
where the inequality is due to that $ f $ is strongly smooth with respect to the constant $ L $ and $ \|\nabla^2_{T_k} f(x^k)\|_2 \le \|\nabla^2 f(x^k)\|_2 \le L $. We will discuss two cases.

\textbf{Case 1:} $ S_k = \emptyset $. From Algorithm \ref{algorithm-NL0R}, we know that $ T_k = T_{k-1} $, thus $ J_k = T_{k-1} \setminus T_K = \emptyset $, and $ d_{\overline{T}_k}^k = -x_{\overline{T}_k}^k = 0 $. Since $ \mu_k \le D $, we obtain that
\begin{align}\label{lem-2-1-6}
2\mu_k - 2n\tau L\mu_k - n\tau \mu_k^2 &\overset{\eqref{2-1-1}}{\ge} 2\mu_k - \frac{2L\nu \mu_k }{2L+D} - \frac{\nu \mu_k^2}{2L+D} \nonumber\\
&\ \ge \ 2\mu_k - \frac{2L\mu_k }{2L+D} - \frac{ \mu_k^2}{2L+D} \nonumber\\
&\ = \ \frac{1}{2L+D}[2\mu_k(2L+D) - 2L\mu_k - \mu_k^2] \nonumber\\
&\ \ge \ \mu_k \ge 0.
\end{align}
Since $ \|d^k_{T_{k}}\|^2 \le \|d^k\|^2 $, and $ d^k $ is obtained from \eqref{direction}, we obtain that
\begin{align}\label{lem-2-1-2}
2\langle g^k, d^k \rangle &\ = \ 2\langle g_{T_k}^k,d_{T_k}^k \rangle - 2 \langle g_{\overline{T}_k}^k,x_{\overline{T}_k}^k \rangle \nonumber \\
&\overset{\mathrm{\eqref{alg-2-1}}}{\le} -2\delta \|d^k\|^2 + \|x^k_{\overline{T}_{k}}\|^2/(2\tau) - 2\mu_k \|d^k_{T_k}\|^2 \nonumber \\
&\ \le \ -2\delta \|d^k\|^2  + n\tau\|g_{T_k}^k\|^2 - \tau\|g_{T_k}^k\|^2 - 2\mu_k \|d^k_{T_k}\|^2\nonumber \\
&\overset{\mathrm{\eqref{lem-2-1-1}}}{\le} -2\delta \|d^k\|^2 -[2\mu_k - n\tau (L+\mu_k)^2]\|d^k_{T_{k}}\|^2 - \tau\|g_{T_k}^2\|^2 \nonumber \\
&\ \le \  -(2\delta - n\tau L^2) \|d^k\|^2 -(2\mu_k - 2n\tau L\mu_k - n\tau \mu_k^2)\|d^k_{T_{k}}\|^2 - \tau\|g_{T_{k-1}}^k\|^2  \nonumber\\
&\overset{\eqref{lem-2-1-6}}{\le} -(2\delta - n\tau L^2) \|d^k\|^2 -\mu_k \|d^k_{T_{k}}\|^2 - \tau\|g_{T_{k-1}}^k\|^2  \nonumber\\
&\overset{\mathrm{\eqref{2-1-1}}}{\le} -2\rho\|d^{k}\|^{2}  - \tau\|g_{T_{k-1}}^k\|^{2} .
\end{align}
Therefore, \eqref{lem-2-1-0} holds.

\textbf{Case 2:} $S_k \neq \emptyset$. For any $i \in S_k = \tilde{T}_k \setminus T_{k-1} = T_k \setminus T_{k-1}$, since $\text{supp}(x^k) \subset T_{k-1}$, we have $x_i^k = 0$. Since $T_k = \tilde{T}_k$, it holds that for any $i \in S_k$,
\begin{align}\label{lem-2-1-3}
|\tau g_i^k|^2 = |x_i^k - \tau g_i^k|^2 \ge 2\tau \lambda > |x_j^k - \tau g_j^k|^2, \ \forall j \in J_k.
\end{align}
Consequently, it holds that
\begin{align*}
&\quad \quad   (|J_k|/|S_k|)\tau^2 \left[ \|g_{T_k}^k\|^2 - \|g_{T_k \cap T_{k-1}}^k\|^2 \right] \\
&\ =\ (|J_k|/|S_k|)\tau^2 \|g_{S_k}^k\|^2 \\
&\overset{\mathrm{\eqref{lem-2-1-3}}}{\ge} |J_k|2\tau \lambda \overset{\mathrm{\eqref{lem-2-1-3}}}{>} \|x_{J_k}^k - \tau g_{J_k}^k\|^2 \\
&\ = \ \|x_{J_k}^k\|^2 - 2\tau \langle x_{J_k}^k, g_{J_k}^k \rangle + \tau^2 \|g_{J_k}^k\|^2 \\
&\overset{\eqref{2-1}}{=} \|x_{\overline{T}_k}^k \|^2 - 2\tau \langle x_{J_k}^k, g_{J_k}^k \rangle + \tau^2 \|g_{J_k}^k\|^2 \\
&\ =\ \|x_{\overline{T}_k}^k \|^2 - 2\tau \langle x_{J_k}^k, g_{J_k}^k \rangle + \tau^2\left[ \|g_{T_{k-1}}^k\|^2 - \|g_{T_k \cap T_{k-1}}^k\|^2 \right] .
\end{align*}
Since ${|J_k|}/{|S_k|} \le n$, we obtain that
\begin{align}\label{lem-2-1-4}
-2\langle x_{J_k}^k, g_{J_k}^k \rangle &\le n\tau \|g_{T_k}^k\|^2 - \tau \|g_{T_{k-1}}^k\|^2 - \frac{1}{\tau}\|x_{\overline{T}_k}^k \|^2 \nonumber\\
&\overset{\eqref{lem-2-1-1}}{\le} n\tau (L + \mu_k)^2 \|d^k_{T_k}\|^2 - \tau\|g_{T_{k-1}}^k\|^2 - \frac{1}{\tau}\|x_{\overline{T}_k}^k \|^2 .
\end{align}
If $d^k$ is obtained from \eqref{direction}, then
\begin{align}\label{lem-2-1-5}
2\langle g^k_{T_k}, d^k_{T_k} \rangle \overset{\eqref{alg-2-1}}{\le} -2\delta \|d^k\|^2 + \frac{1}{2\tau}\|x^k_{\overline{T}_{k}}\|^2 - 2\mu_k \|d^{k}_{T_{k}}\|^2 .
\end{align}
By calculations, it holds that
\begin{align*}
2\langle g^k, d^k \rangle &\quad \ \, =  \quad\ \,  2\langle g^k_{T_{k}}, d^k_{T_{k}} \rangle - 2\langle g^k_{\overline{T}_k}, d^k_{\overline{T}_k} \rangle \\
&\quad \overset{\eqref{2-1}}{=} \quad 2\langle g^k_{T_{k}}, d^k_{T_{k}} \rangle - 2\langle g^k_{J_k}, d^k_{J_k} \rangle \\
&\overset{\eqref{lem-2-1-4},\eqref{lem-2-1-5}}{\le} -2\delta \|d^k\|^2 - (2\mu_k - n\tau (L+\mu_k)^2) \|d^k_{T_{k}}\|^2 - \tau\|g_{T_{k-1}}^k\|^2 - \|x_{\overline{T}_k}^k \|^2/(2\tau) \\
&\quad \  \le \quad\,\; -(2\delta - n\tau L^2) \|d^k\|^2 -(2\mu_k - 2n\tau L\mu_k - n\tau \mu_k^2)\|d^k_{T_{k}}\|^2 - \tau\|g_{T_{k-1}}^k\|^2  \\
&\quad \overset{\eqref{lem-2-1-6}}{\le} \ \,\; -(2\delta - n\tau L^2) \|d^k\|^2 -\mu_k \|d^k_{T_{k}}\|^2 - \tau\|g_{T_{k-1}}^k\|^2  \\
&\quad \overset{\eqref{2-1-1}}{\le} \ \ -2\rho\|d^{k}\|^{2} - \tau\|g_{T_{k-1}}\|^{2}.
\end{align*}
If $d_{T_k}^k = - g_{T_k}^k$, it holds that
\begin{align*}
2\langle g^k, d^k \rangle &\ = \quad\,  2\langle g^k_{T_{k}}, d^k_{T_{k}} \rangle - 2\langle g^k_{\overline{T}_k}, d^k_{\overline{T}_k} \rangle \\
&\ = \quad  -2\|d_{T_k}^k \|^2 - 2\langle g^k_{J_k}, d^k_{J_k} \rangle \\
&\overset{\eqref{lem-2-1-4}}{\le} \,\; -2\|d_{T_k}^k \|^2 +n\tau \|g_{T_k}^k\|^2 - \tau \|g_{T_{k-1}}^k\|^2 - \|x_{\overline{T}_k}^k \|^2/\tau\\
&\overset{\eqref{2-1}}{=} \,\; -(2-n\tau )\|d_{T_k}^k \|^2 - \|d_{\overline{T}_k}^k \|^2/\tau -  \tau \|g_{T_{k-1}}^k\|^2 \\
&\ \le \quad  -(2-n\tau )(\|d_{T_k}^k \|^2 + \|d_{\overline{T}_k}^k \|^2) -  \tau \|g_{T_{k-1}}^k\|^2 \\
&\overset{\eqref{2-1-1}}{\le} \,\;  -2\rho \|d^k\|^2 - \tau \|g_{T_{k-1}}^k\|^2,
\end{align*}
where the second inequality is from $-1/\tau \le \tau - 2 \le n\tau -2$ for any $\tau > 0$. Thus, \eqref{lem-2-1-0} holds, the proof is complete.
\end{proof}
Our next result shows that $\alpha_k$ exists and is bound away from zero.
\begin{lemma}\label{lem-2-2}
{\rm{(Existence and boundedness of $\alpha_k$)}} Let $f$ be strongly smooth with respect to the constant $L > 0$, and let $\overline{\tau}, \overline{\alpha}$ be defined as above. Then, for any $k \ge 0$ and any parameters
$0 < \alpha \le \overline{\alpha}$,  $0 < \delta \le \min \{1, 2L\}$,  $0 < \tau \le \min \left\{\frac{\alpha \delta}{nL^2}, \frac{\alpha}{n}, \frac{1}{4L}, \frac{\nu}{n(2L+D)}\right\}$,  $0 < \nu \le \frac{\alpha(1-\sigma)}{1-\alpha \sigma},
$
the following inequality holds:
\begin{align}\label{lem-2-2-1}
f(x^k(\alpha))  \le f(x^k) + \sigma \alpha \langle g^k, d^k \rangle.
\end{align}
Moreover, for any $\tau \in (0, \overline{\tau})$, we have $\liminf\limits_{k \to \infty} \{\alpha_k\} \ge \beta \overline{\alpha} > 0$.
\end{lemma}
\begin{proof}
If $0 < \alpha \le \overline{\alpha}$ and $0 < \delta \le \min \{1, 2L\}$, then
\begin{align}\label{lem-2-2-2}
\alpha \le \frac{2(1-\sigma)\delta}{L}, \quad \alpha \le \frac{1-2\sigma}{L/\delta - \sigma} \le \frac{1-2\sigma}{\max\{0, L - \sigma\}}.
\end{align}
Since $f$ is strongly smooth, it holds that
\begin{align*}
&\quad\quad\  2f(x^k(\alpha)) - 2f(x^k) - 2\alpha \sigma \langle g^k, d^k \rangle \\
&\overset{\eqref{eq-2-4}}{\le} 2\langle g^k, x^k(\alpha) - x^k \rangle + L\|x^k(\alpha) - x^k\|^2 - 2\alpha \sigma \langle g^k, d^k \rangle \\
&\overset{\eqref{2.4}}{=} \alpha (1-\sigma)2\langle g^k_{T_k}, d^k_{T_k} \rangle - (1-\alpha \sigma)2\langle g^k_{\overline{T}_k}, d^k_{\overline{T}_k} \rangle + L[\alpha^2\|d^k_{T_k}\|^2 + \|x^k_{\overline{T}_k}\|^2 \\
&\ = \ \, \alpha (1-\sigma)2\langle g^k_{T_k}, d^k_{T_k} \rangle - (1-\alpha \sigma)2\langle g^k_{J_k}, d^k_{J_k} \rangle + L[\alpha^2\|d^k_{T_k}\|^2 + \|x^k_{\overline{T}_k}\|^2 =: \psi .
\end{align*}
We need to prove that $\psi \le 0$. Consider two cases.

\textbf{Case 1:} $S_k = \emptyset$. From the algorithm, we have $T_k = T_{k-1}$, thus $J_k = T_{k-1} \setminus T_k = \emptyset$. Therefore, we obtain that
\begin{align}\label{lem-2-2-3}
\psi &= \alpha (1-\sigma)2\langle g^k_{T_k}, d^k_{T_k} \rangle + L\alpha^2\|d^k_{T_k}\|^2 \nonumber \\
&\quad \begin{cases}
\overset{\eqref{alg-2-1}}{\le} -2\alpha (1-\sigma)\delta \|d^k\|^2 -2\alpha (1-\sigma)\mu_k \|d^k_{T_k}\|^2 + L\alpha^2\|d^k_{T_k}\|^2 \nonumber \\
\overset{\eqref{alg-2-3}}{\le} -2\alpha (1-\sigma) \|d^k_{T_k}\|^2 + L\alpha^2\|d^k_{T_k}\|^2
\end{cases} \nonumber \\
&\ \;  \le \ \, -2\alpha (1-\sigma)\delta \|d^k\|^2 + L\alpha^2\|d^k_{T_k}\|^2 \nonumber \\
&\ \;  = \ \, \alpha (L\alpha - 2(1-\sigma)\delta)\|d^k\|^2 \quad (\text{by}\  \|d^k\|^2 = \|d^k_{T_k}\|^2) \nonumber \\
&\overset{\eqref{lem-2-2-2}}{\le} 0.
\end{align}

\textbf{Case 2:} $S_k \neq \emptyset$. If $d^k$ from \eqref{direction}, by \eqref{lem-2-1-4} and \eqref{lem-2-1-5}, it holds that
\begin{align*}
\psi &\ \le \ \alpha (1-\sigma)[-2\delta \|d^k\|^2 - 2\mu_k\|d^k_{T_k}\|^2 + (1/2\tau)\|x^k_{T_k}\|^2] + L\alpha^2\|d^k_{T_k}\|^2 \\
&\quad \ \   + (1-\alpha \sigma)[n\tau (L+
\mu_k)^2 \|d^k_{T_k}\|^2 - \tau\|g_{T_{k-1}}^k\|^2 - \|x_{\overline{T}_k}^k \|^2/\tau] + L\|x_{\overline{T}_k}^k \|^2 \\
&\  \le \  -2\delta \alpha(1-\sigma)\|d^k\|^2 + (1-\alpha \sigma)n\tau L^2\|d^k\|^2 + L\alpha^2 \|d^k\|^2 \\
&\quad \ \   + \frac{\alpha(1-\sigma)}{2\tau}\|x_{\overline{T}_k}^k \|^2 - \frac{1-\alpha \sigma}{\tau}\|x_{\overline{T}_k}^k \|^2  + L\|x_{\overline{T}_k}^k \|^2 \\
&\quad \ \  + \|d^k_{T_k}\|^2[ -2\mu_k \alpha(1-\sigma) + (1-\alpha \sigma)n\tau (2\mu_k L + \mu_k^2)]\\
&\   \quad\ -(1-\alpha \sigma)\tau \|g_{T_{k-1}}^k\|^2 \\
&\   \le \  c_1 \|d^k\|^2 + c_2\|x_{\overline{T}_k}^k \|^2 + c_3\mu_k \|d^k_{T_k}\|^2 - (1-\alpha \sigma)\tau \|g_{T_{k-1}}^k\|^2 \\
&\ \le \  c_1 \|d^k\|^2 + c_2\|x_{\overline{T}_k}^k \|^2 + c_3\mu_k \|d^k_{T_k}\|^2.
\end{align*}
Because $\alpha \in (0,1)$, $\sigma \in (0,1/2]$, it holds that $1-\alpha \sigma > 0$.
\begin{align*}
c_1 &:= -2\delta \alpha(1-\sigma) + (1-\alpha \sigma)n\tau L^2 + L\alpha^2  \\
&\le \ -\alpha(1-\sigma)2\delta + (1-\alpha \sigma)\delta \alpha + L\alpha^2  \quad (\text{by}\ 1-\alpha \sigma > 0,\  \tau \le \alpha \delta / (nL^2) )\\
&= \ \alpha [(L-\sigma \delta)\alpha - (1-2\sigma)\delta] \le 0. \quad(\text{by}\ \eqref{lem-2-2-2},\ L-\alpha \sigma > 0,\  1-2\sigma > 0)
\end{align*}
and
\begin{align*}
c_2 &:= \alpha (1-\sigma)/(2\tau) - (1-\alpha \sigma)/\tau + L \\
&\le \, (1-\alpha \sigma)/(2\tau) - (1-\alpha \sigma)/\tau + L \quad (\text{by}\  1-\alpha \sigma > 0 )\\
&\le \, -(1-\alpha \sigma)/(2\tau) + L \le 0. \quad (\text{by}\  1-\alpha \sigma > 0,\  \tau \le 1/(4L))
\end{align*}
Moreover, it holds that
\begin{align*}
c_3 &\ := \ -2 \alpha(1-\sigma) + (1-\alpha \sigma)n\tau (2 L + \mu_k) \\
&\overset{\eqref{2-1-1}}{\le}  -2 \alpha(1-\sigma) + \frac{\nu(1-\alpha \sigma)(2 L + \mu_k)}{2L + D} \\
&\ \le \ \, -2 \alpha(1-\sigma) + \frac{\nu(1-\alpha \sigma)(2 L + D)}{2L + D} \\
&\ = \ -2 \alpha(1-\sigma) +\nu(1-\alpha \sigma) \\
&\  \le \ \; -2 \alpha(1-\sigma) +(1-\alpha \sigma)\frac{\alpha(1-\sigma)}{1-\alpha \sigma} \quad (\text{by}\ 1-\alpha \sigma >0,\  0< \nu \le \frac{\alpha(1-\sigma)}{1-\alpha \sigma} )\\
&\ = \  -\alpha + \alpha \sigma \le 0.
\end{align*}
If $d^k$ is the gradient step, that is $d^k_{T_k} = -g^k_{T_k}$ and $d^k_{\overline{T}_k} = -x^k_{\overline{T}_k}$, it holds that
\begin{align*}
\psi &\overset{\eqref{lem-2-1-4}}{\le} -2\alpha(1-\sigma) \|d^k_{T_k}\|^2 + L\alpha^2 \|d^k_{T_k}\|^2 \quad (\text{by}\  1-2\alpha \sigma > 0 )\\
&\quad \quad +  (1-\alpha \sigma)[n\tau \|g_{T_k}^k\|^2 - \tau \|g_{T_{k-1}}^k\|^2 - \|x_{\overline{T}_k}^k \|^2/\tau] + L\|x_{\overline{T}_k}^k \|^2 \\
&\ \le \ \,  c_4 \|d^k_{T_k}\|^2 + c_5\|x_{\overline{T}_k}^k \|^2 - (1-\alpha \sigma)\tau \|g_{T_{k-1}}^k\|^2 ,
\end{align*}
where $c_4$, $c_5$ be defined as
\begin{align*}
c_4 &:= -2\alpha(1-\sigma) + (1-\alpha \sigma)n \tau + L \alpha^2 \\
&\le -2\alpha(1-\sigma) + (1-\alpha \sigma)\alpha + L \alpha^2 \quad (\text{by}\ 1-\alpha \sigma >0, \tau \le \alpha/n) \\
&= \alpha[ (L-\sigma)\alpha - (1-2\sigma)] \\
&\le \alpha \left[\max \{0,L-\sigma \} \frac{1-2\sigma}{\max \{0, L - \sigma\}} - (1-2\sigma) \right], \quad (\text{by}\ 1-2\sigma > 0 \ \eqref{lem-2-2-2} ) \\
&\le \alpha \left[\max \{0,L-\sigma \}\frac{1-2\sigma}{\max \{0, L - \sigma\}} - (1-2\sigma) \right] \le 0.\quad (\text{by}\ 1-2\sigma > 0) \\
c_5 &:= -(1-\alpha \sigma)/\tau + L\\
&\le -1/(2\tau) + L\le 0. \  (\text{by }\ 1-\alpha \sigma \ge 1/2, \tau \le 1/(4\tau))
\end{align*}
Therefore, \eqref{lem-2-2-1} holds. If $\tau \in (0, \overline{\tau})$, then for any $\alpha \in [\beta \overline{\alpha}, \overline{\alpha}]$,  it holds that
\begin{align*}
0 < \tau &\overset{\eqref{2-1-1}}{<} \min \left\{ \overline{\alpha}\delta \beta /(nL^2), \overline{\alpha}\beta /n, 1/(4L), \frac{\nu}{n(2L+D)} \right\} \\
&\ \le\ \, \min \left\{  \frac{\alpha \delta}{nL^2},  \frac{\alpha}{n}, \frac{1}{4L}, \frac{\nu}{n(2L+D)} \right\},
\end{align*}
and $0< \nu \le \frac{\overline{\alpha} \beta(1-\sigma)}{1-\overline{\alpha} \beta\sigma} \le \frac{\alpha(1-\sigma)}{1-\alpha \sigma}$. Therefore, \eqref{lem-2-2-1} holds for any $\alpha \in [\beta \overline{\alpha}, \overline{\alpha}]$. According to the Armijo step size criterion, the sequence $\{\alpha_k\}$ has a lower bound of $\beta \overline{\alpha}$, that is
\begin{align}\label{lem-2-2-4}
\liminf\limits_{k\to \infty}\{\alpha_k \} \ge \beta \overline{\alpha} > 0.
\end{align}
The Lemma is proved.
\end{proof}

\begin{lemma}\label{lem-2-3}
Let $f$ be strongly smooth with respect to the constant $L > 0$, and let $\overline{\tau}$ be defined as above. The sequence $\{x^k\}$ is generated by MBNL0R, with $\tau \in (0, \overline{\tau})$ and $\delta \in (0, \min \{1, 2L\})$. Then, the sequence $\{f(x^k)\}$ is strictly decreasing, and
\begin{align}\label{lem-2-3-1}
\lim\limits_{k\to\infty} \max \{\|F_{\tau}(x^k; T_k)\|, \|x^{k+1}-x^{k}\|, \|g_{T_{k-1}}^{k}\|, \|g_{T_k}^k\|\} =0.
\end{align}
\end{lemma}
\begin{proof}
By \eqref{lem-2-2-1}, \eqref{lem-2-1-0} and denoting $c_{0}:=\sigma \overline{\alpha} \beta \rho$, we obtain that
\begin{align*}
f(x^{k+1})-f(x^{k}) &\ \leq \ \  \sigma \alpha_{k}\langle g^{k}, d^{k}\rangle \\
&\stackrel{\eqref{lem-2-1-0}}{\leq}-\sigma \alpha_{k} \rho\|d^{k}\|^{2}-\frac{\tau}{2}\|g_{T_{k-1}}^{k}\|^{2} \\
&\stackrel{\eqref{lem-2-2-4}}{\leq}-c_{0}\|d^{k}\|^{2}-\frac{\tau}{2}\|g_{T_{k-1}}^{k}\|^{2} .
\end{align*}
Then, it follows from the above inequality that
$$
\begin{aligned}
\sum_{k=0}^{\infty}\left[c_{0}\|d^{k}\|^{2}+\frac{\tau}{2}\|g_{T_{k-1}}^{k}\|^{2}\right] & \leq \sum_{k=0}^{\infty}\left[f(x^{k})-f(x^{k+1})\right] \\
& =\left[f(x^{0})-\lim _{k \to \infty} f(x^{k})\right]<+\infty ,
\end{aligned}
$$
where the last inequality is due to $f$ being bounded from below. Hence $\|d^{k}\| \to 0$, $\|g_{T_{k-1}}^{k}\| \to 0$, which suffices to obtain that $\|x^{k+1}-x^{k}\| \to 0$ because of
$$
\|x^{k+1}-x^{k}\|^{2} \stackrel{\eqref{2.4}}{=} \alpha_{k}^{2}\|d_{T_{k}}^{k}\|^{2}+\|x_{\overline{T}_{k}}^{k}\|^{2} \leq\|d_{T_{k}}^{k}\|^{2}+\|d_{\overline{T}_{k}}^{k}\|^{2}=\|d^{k}\|^{2}.
$$
The above relation also indicates that $\|x_{\overline{T}_{k}}^{k}\|^{2} \rightarrow 0$. In addition, if $d^{k}$ is taken from \eqref{direction}, then $\|g_{T_{k}}^{k}\| \leq (L+\mu_k)\|d^{k}_{T_k}\| \to 0$ by \eqref{lem-2-1-1}. If it is taken from \eqref{alg-2-3} then $\|g_{T_{k}}^{k}\|=$ $\|d_{T_{k}}^{k}\| \to 0$. These results together with \eqref{eq-2-7} imply that $\|F_{\tau}\left(x^{k} ; T_{k}\right)\|^{2}=\|g_{T_{k}}^{k}\|^{2}+$ $\|x_{\overline{T}_{k}}^{k}\|^{2} \to 0$. The proof is complete.
\end{proof}

It can be conclude from Lemma \ref{lem-2-3} that the index set $T$ can be identified within finite steps and the sequence converges to a $\tau$-stationary point or a local minimizer globally that are presented by the following theorem.

\begin{theorem}{\rm{\cite[Theorem 3]{zhou2021newton}}
(Convergence and identification of $T_{k}$ )}\label{the-2-1} Let $f$ be strongly smooth with respect to the constant $L > 0$, and let $\overline{\tau}$ be defined as above. The sequence $\{x^k\}$ is generated by MBNL0R, with $\tau \in (0, \overline{\tau})$ and $\delta \in (0, \min \{1, 2L\})$. Then
\begin{description}
%\begin{itemize}
\item{(i)} For any sufficiently large $k$, $T_k \equiv T_{k-1} \equiv T_{\infty}$.
\item{(ii)} Any limit point $x^{*}$ satisfies
\begin{align}\label{the-2-1-1}
\nabla_{T_{\infty}}f(x^*) = 0,\quad x_{\overline{T}_{\infty}} = 0,\quad \text{supp}(x^*) \subset T_{\infty},
\end{align}
and $x^{*} \neq 0$. Furthermore, if $\tau_{*}$ satisfies
\begin{align}\label{the-2-1-2}
0 < \tau_{*} < \min \{ \overline{\tau}, \min\limits_{i \in \text{supp}(x^*)} |x_{i}^{*}|/(2\lambda)\},
\end{align}
then $x^{*}$ is a $\tau_{*}$-stationary point.
\item{(iii)} If $x^{*}$ is an isolated point, then the entire iteration sequence converges to $x^{*}$.
%\end{itemize}
\end{description}
\end{theorem}
The proof of Theorem \ref{the-2-1} is the same as that in \cite[Theorem 3]{zhou2021newton}. For brevity, the proof is omitted here.

\begin{theorem}\label{the-2-2}
{\rm{(Global convergence)}} Let $ \{x^{k} \}$ be the sequence generated by MBNL0R with $\tau \in(0, \overline{\tau})$ and $\delta \in (0, \min  \{1, \ell_{*} \} )$ and $x^{*}$ be one of its accumulating points. Suppose $f$ is strongly smooth with constant $L>0$ and locally strongly convex with $\ell_{*}>0$ around $x^{*}$. Then, the following results hold.

(i) The whole sequence converges to $x^{*}$, a strictly local minimizer of \eqref{L0}.

(ii) The Newton direction is always accepted for sufficiently large $k$.
\end{theorem}
\begin{proof}
(i) Denote $T_{*}:=\operatorname{supp} (x^{*} )$. Theorem \ref{the-2-1} shows that $\nabla_{T_{*}} f (x^{*} )=0$ and $x^{*} \neq 0$. Consider a local region $N (x^{*} ):= \{x \in \mathbb{R}^{n}:\|x-x^{*}\|<\epsilon_{*} \}$, where
$$
\epsilon_{*}:=\min  \left\{\lambda / (2\|\nabla_{\overline{T}_{*}} f (x^{*} )\| ), \min _{i \in T_{*}} |x_{i}^{*} | \right\}.
$$
For any $x (\neq x^{*} ) \in N (x^{*} )$, we have $T_{*} \subseteq \operatorname{supp}(x)$. In fact if there is a $j$ such that $j \in T_{*}$ but $j \notin \operatorname{supp}(x)$, then we derive a contradiction:
$$
\epsilon_{*} \leq \min _{i \in T_{*}} |x_{i}^{*} | \leq |x_{j}^{*} |= |x_{j}^{*}-x_{j} | \leq\|x-x^{*}\|_{2}<\epsilon_{*} .
$$
Since $f$ is locally strongly convex with $\ell_{*}>0$ around $x^{*}$, for any $x (\neq x^{*} ) \in N (x^{*} )$, it is true that
$$
\begin{aligned}
& f(x)+\lambda\|x\|_{0}-f (x^{*} )-\lambda\|x^{*}\|_{0} \\
\geq &  \langle\nabla f (x^{*} ), x-x^{*} \rangle+ (\ell_{*} / 2 )\|x-x^{*}\|^{2}+\lambda\|x\|_{0}-\lambda\|x^{*}\|_{0} \\
= &  \langle\nabla_{\overline{T}_{*}} f (x^{*} ), x_{\overline{T}_{*}} \rangle+ (\ell_{*} / 2 )\|x-x^{*}\|^{2}+\lambda\|x\|_{0}-\lambda\|x^{*}\|_{0}=: \phi .
\end{aligned}
$$
where the first equality is due to $\nabla_{T_{*}} f (x^{*} )=0$. Clearly, if $T_{*}=\operatorname{supp}(x)$, then $x_{\overline{T}_{*}}=0,\|x\|_{0}=\|x^{*}\|_{0}$ and hence $\phi= (\ell_{*} / 2 )\|x-x^{*}\|^{2}>0$. If $T_{*} \neq(\subseteq) \operatorname{supp}(x)$, then $\|x\|_{0} \geq\|x^{*}\|_{0}+1$ and thus we obtain
$$
\begin{aligned}
\phi & \geq-\|\nabla_{\overline{T}_{*}} f (x^{*} )\|\|x_{\overline{T}_{*}}\|+ (\ell_{*} / 2 )\|x-x^{*}\|^{2}+\lambda \\
& \geq-\|\nabla_{\overline{T}_{*}} f (x^{*} )\|\|x-x^{*}\|+ (\ell_{*} / 2 )\|x-x^{*}\|^{2}+\lambda \\
& \geq-\lambda / 2+ (\ell_{*} / 2 )\|x-x^{*}\|^{2}+\lambda>0.
\end{aligned}
$$
Both cases show that $x^{*}$ is a strictly local minimizer of \eqref{L0} and is unique in $N (x^{*} )$, namely, $x^{*}$ is isolated local minimizer in $N (x^{*} )$. Therefore, the whole sequence tends to $x^{*}$ by Theorem \ref{the-2-1}).

(ii) We first verify that $Q_{k}$ is nonsingular when $k$ is sufficiently large and
$$
2  \langle g_{T_{k}}^{k}, d_{T_{k}}^{k} \rangle \leq -2\delta \|d^k\|^2 + \|x^k_{\overline{T}_{k}}\|/(2\tau) - 2\mu_k \|d^{k}_{\overline{T}_{k}}\|^2 .
$$
Since $f$ is strongly smooth with $L$ and locally strongly convex with $\ell_{*}$ around $x^{*}$, it holds that
\begin{align}\label{the-2-2-2}
\ell_{*} \leq  \lambda_{i} (\nabla^2_{T_k}f(x^k) ) \leq L ,\ \ell_{*} \leq  \lambda_{i} (Q_{k} ) \leq L + \mu_k,
\end{align}
where $\lambda_{i}(A)$ is the $i$ th largest eigenvalue of $A$. Direct verification yields that
$$
\begin{aligned}
2 \langle g_{T_{k}}^{k}, d_{T_{k}}^{k} \rangle &\, \stackrel{\eqref{direction}}{=} \, -2\langle (Q_{k} + \mu_k I)d^k_{T_{k}}, d^k_{T_{k}} \rangle \\
& \stackrel{\eqref{the-2-2-2}}{\leq}-2\ell_{*} \|d_{T_{k}}^{k}\|^{2} - 2\mu_k \|d_{T_{k}}^{k}\|^{2} \\
&\ \; = \ \; -2\ell_{*}(\|d^k_{T_k}\|^2 + \|d^k_{J_k}\|^2 - \|d^k_{J_k}\|^2) - 2\mu_k \|d^k_{T_k}\|^2 \\
&\ \; = \ \; -2 \ell_{*}\|d_{T_{k} \cup J_{k}}^{k}\|^{2}+2\ell_{*}\|d_{J_{k}}^{k}\|^{2} - 2\mu_k \|d^k_{T_k}\|^2 \\
&\, \stackrel{\eqref{2-1}}{=} \, -2 \ell_{*}\|d^{k}\|^{2} + 2\ell_{*}\|x_{\overline{T}_{k}}^{k}\|^{2} - 2\mu_k \|d^k_{T_k}\|^2\\
&\ \; \le \ \; -2 \ell_{*}\|d^{k}\|^{2}+2 L\|x_{\overline{T}_{k}}^{k}\|^{2} - 2\mu_k \|d^k_{T_k}\|^2 \\
&\ \; \le \ \; -2 \delta\|d^{k}\|^{2}+\|x_{\overline{T}_{k}}^{k}\|^{2} /(2 \tau) - 2\mu_k \|d^k_{T_k}\|^2,
\end{aligned}
$$
where the last inequality is due to $\delta \leq \ell_{*}$ and $\tau<\overline{\tau} \leq 1 /(4 L)$. This proves that $d^{k}$ from \eqref{direction} is always admitted for sufficiently large $k$. The proof is finished.
\end{proof}
If the regularization parameter $\mu_k = 0$, MBNL0R will ultimately have the same iteration format as NL0R \cite{zhou2021newton}. This means that our proposed algorithm not only requires less matrix computation, but can also ultimately achieve the same quadratic convergence rate as NL0R \cite{zhou2021newton}. These will be demonstrated by the following Theorem.
\begin{theorem}
{\rm{(Quadratic convergence)}} Let $ \{x^{k} \}$ be the sequence generated by MBNL0R with $\tau \in(0, \overline{\tau})$ and $\delta \in (0, \min  \{1, \ell_{*} \} )$ and $x^{*}$ be one of its accumulating points. Suppose $f$ is strongly smooth with constant $L>0$ and locally strongly convex with $\ell_{*}>0$ around $x^{*}$. Furthermore, if the Hessian of $f$ is locally Lipschitz continuous around $x^*$ with constant $M_* > 0$. Then, there exist $C > 0.5M_*$ such that
\begin{align}\label{eq-quadraticly}
\|x^{k+1} - x^*\| \le C/\ell_*\|x^k - x^*\|^2.
\end{align}
\end{theorem}
\begin{proof}
From \eqref{support} and claim (i) in Theorem \ref{the-2-1}, we observe that $x^k_{T_k} \equiv 0$ for sufficiently large $k$.
Without loss of generality, we assume that $x^k_{T_k}=0$ in this proof. Hence, the direction update scheme in \eqref{direction} is equivalent to the following form.
\begin{align}\label{eq-zhounewton}
\begin{cases}
(\nabla_{T_k}^{2}f(x^k) + \mu_k I)d_{T_k}^k &= \nabla^2_{T_k, \overline{T}_k}f(x^k)x_{\overline{T}_k}^k - \nabla_{T_k}f(x^k) , \\
d_{\overline{T}_k}^k &= -x_{\overline{T}_k}^k.
\end{cases}
\end{align}
Also by Theorem \ref{the-2-1}, for sufficiently large $k$, we have \eqref{the-2-1-1}, namely
\begin{align}\label{the-2-2-3}
g_{T_{k}}^* = 0,\quad x_{\overline{T}_{k}}^* = 0.
\end{align}
For any $0 \leq t \leq 1$, by letting $x(t):=x^{*}+t\left(x^{k}-x^{*}\right)$. The Hessian of $f$ being locally Lipschitz continuous at $x^{*}$ derives
\begin{align}\label{the-2-2-4}
\left\|\nabla_{T_{k}:}^{2} f\left(x^{k}\right)-\nabla_{T_{k}:}^{2} f(x(t))\right\|_{2} \leq M_{*}\left\|x^{k}-x(t)\right\|=(1-t) M_{*}\left\|x^{k}-x^{*}\right\| .
\end{align}
Moreover, by Taylor expansion, we obtain
\begin{align}\label{the-2-2-5}
\nabla f(x^k) - \nabla f(x^* ) =
\int_{0}^{1} \nabla^2 f(x(t) )(x^k - x^*)dt .
\end{align}
By \eqref{support} and \eqref{the-2-2-3}, there is $\| x^{k+1}_{\overline{T}_k} - x^*_{\overline{T}_k} \|=0$. Together with \eqref{2.4}, it holds that
\begin{align}\label{eq-xk1alpha}
\| x^{k+1} - x^* \|^2 &= \| x^{k+1}_{T_k} - x^*_{T_k} \|^2 + \| x^{k+1}_{\overline{T}_k} - x^*_{\overline{T}_k} \|^2 \nonumber \\
&= \| x^{k}_{T_k} - x^*_{T_k} + \alpha_k d_{T_k}^k \|^2  .
\end{align}
By \eqref{eq-xk1alpha} , we have the following chain of inequalities
\begin{align}\label{eq-3.37}
\| x^{k+1} - x^* \|^2 &\ \, = \ \, \| x^{k}_{T_k} - x^*_{T_k} + \alpha_k d_{T_k}^k \|^2 \nonumber \\
&\ \, = \ \, \|(1 - \alpha_k)(x^{k}_{T_k} - x^*_{T_k}) + \alpha_k( x^{k}_{T_k} - x^*_{T_k} +  d_{T_k}^k ) \|^2 \nonumber \\
&\ \, \le \ \, (1 - \alpha_k)\| x^{k}_{T_k} - x^*_{T_k} \|^2 + \alpha_k \| x^{k}_{T_k} - x^*_{T_k} +  d_{T_k}^k \|^2 .
\end{align}
where the last inequality is due to $\| \cdot \|^2$ being a convex function.
Combine with \eqref{eq-3.37} and \eqref{lem-2-2-4}, we obtain that
\begin{align}\label{eq-3.38}
\| x^{k+1} - x^* \|^2 \le (1 - \beta \overline{\alpha})\| x^{k} - x^* \|^2 + \overline{\alpha} \| x^{k}_{T_k} - x^*_{T_k} +  d_{T_k}^k \|^2.
\end{align}
Denote $Q_k = \nabla_{T_k}^2 f(x^k) + \mu_k I$. From claim (ii) in Theorem \ref{the-2-2}, $d^k$ will be updated eventually by \eqref{eq-zhounewton}. Hence, by \eqref{eq-zhounewton} and \eqref{the-2-2-2}, for sufficiently large $k$, it holds that
\begin{align}\label{eq-lxk-1}
&\quad \ {\ell_*}  \|x_{ T_k}^{k}-x_{ T_k}^{*}+d_{ T_k}^{k} \| \nonumber\\
&=\ell_*  \| Q_k^{-1}[\nabla_{T_k,\overline{T}_k}^{2}f(x^k)x_{\overline{T}_k}^k - g_{T_k}^k] + x_{T_k}^k - x_{T_k}^*  \| \nonumber\\
&\le \|\nabla_{T_k,\overline{T}_k}^{2}f(x^k)x_{\overline{T}_k}^k - g_{T_k}^k + Q_k(x_{T_k}^k - x_{T_k}^*) \| .
\end{align}
Denote $z^k = x^k - \tau \nabla f(x^k)$, by $g_{T_{k}}^* = 0$ and $x_{\overline{T}_{k}}^* = 0$ from  \eqref{the-2-2-3}, it holds that
\begin{align}\label{eq-lxk-2}
&\quad \, \nabla_{T_k,\overline{T}_k}^{2}f(x^k)x_{\overline{T}_k}^k - g_{T_k}^k + (\nabla_{T_k, T_k}^2f(x^k) + \mu_k I)(x_{T_k}^k - x_{T_k}^*)  \nonumber\\
&= \nabla_{T_k :}^2f(x^k)x^k - \nabla_{T_k, T_k}^2f(x^k) x_{T_k}^* - g_{T_k}^k + \mu_k I(x_{T_k}^k - x_{T_k}^*) \nonumber\\
&= \nabla_{T_k :}^2f(x^k)x^k - \nabla_{T_k :}^2f(x^k) x^* - g_{T_k}^k  + \nabla_{T_k, \overline{T}_k}^2f(x^k) x^*_{\overline{T}_k} + \mu_k I(x_{T_k}^k - x_{T_k}^*) \nonumber\\
&= \nabla_{T_k :}^2f(x^k)x^k - \nabla_{T_k :}^2f(x^k) x^* - g_{T_k}^k + g_{T_k}^* + \mu_k I(x_{T_k}^k - x_{T_k}^*), \ (\mbox{by}\ \eqref{the-2-2-3}).
\end{align}
Combine with \eqref{eq-lxk-1} and \eqref{eq-lxk-2}, we obtain that
\begin{align}\label{eq-xdtheta}
&\quad\ {\ell_*}  \|x_{ T_k}^{k}-x_{ T_k}^{*}+d_{ T_k}^{k} \| \nonumber\\
&\le \| \nabla_{T_k :}^{2}f(x^k)x^k - g_{T_k}^k  - \nabla_{T_k :}^{2}f(x^k)x^* + g_{T_k}^* \| + \mu_k \|x_{T_k}^k - x_{T_k}^*\|.
\end{align}
By \eqref{eq-xdtheta} and \eqref{the-2-2-5}, it holds that
\begin{align*}
&\quad\  {\ell_*} \left\|x_{ T_k}^{k}-x_{ T_k}^{*}+d_{ T_k}^{k}\right\| \\
&\le \left\| \nabla_{T_k :}^{2}f(x^k)x^k - g_{T_k}^k  - \nabla_{T_k :}^{2}f(x^k)x^* + g_{T_k}^* \right\| + \mu_k \|x_{T_k}^k - x_{T_k}^*\|\\
&\le \left\|  \int_{0}^{1}[ \nabla_{T_k :}^2 f(x^k) - \nabla_{T_k :}^2 f(x(t))] (x^k - x^*) dt \right\| + \mu_k \|x_{T_k}^k - x_{T_k}^*\|\\
&\le \int_{0}^{1}\|  \nabla_{T_k,:}^2 f(x^k) - \nabla_{T_k,:}^2 f(x(t))\| \|(x^k - x^*) \|dt + \mu_k \|x_{T_k}^k - x_{T_k}^*\|\\
&\le M_{*}\| x^k - x^*\|^2 \int_{0}^{1}(1-t)dt + \mu_k \|x_{T_k}^k - x_{T_k}^*\|\\
&\le 0.5 M_{*}\| x^k - x^*\|^2 + \mu_k \|x_{T_k}^k - x_{T_k}^*\|.
\end{align*}
By Algorithm \ref{algorithm-NL0R}, there is $\mu_k \le \|F_{\tau}(x^k; T_k)\|^2$, together with \eqref{the-2-2-3}, we obtain that
\begin{align*}
\sqrt{\mu_k}\le
\begin{Vmatrix}
g_{T_k}^k - g_{T_k}^* \\
x_{\overline{T}_k}^k - x_{\overline{T}_k}^*
\end{Vmatrix}
= \|g_{T_k}^k - g_{T_k}^*\| \le L\|x_{T_k}^k - x_{T_k}^*\| .
\end{align*}
Combine with $\|F_{\tau}(x^k; T_k)\| \to 0$, there exits $C > 0.5M_*$ such that
\begin{align*}
\ell_* \left\|x_{ T_k}^{k}-x_{ T_k}^{*}+d_{ T_k}^{k}\right\| \le C \|x^k - x^*\|^2,
\end{align*}
holds for sufficiently large $k$. It follows from $d_{\overline{T}_k}^k = - x_{\overline{T}_k}^k$ and \eqref{the-2-2-3}, we can see that $\|x^k + d^k - x^*\| = \|x_{ T_k}^{k}-x_{ T_k}^{*}+d_{ T_k}^{k}\|$, leading to the following fact
\begin{align}\label{eq-3.41}
\frac{\|x^k + d^k - x^*\|}{\|x^k - x^*\|} = \frac{\|x_{ T_k}^{k}-x_{ T_k}^{*}+d_{ T_k}^{k}\|}{\|x^k - x^*\|} \le \frac{C \|x^k - x^*\|^2}{ {\ell_*} \|x^k - x^*\|} \le \frac{C}{\ell_*}\|x^k - x^*\|.
\end{align}
Now, we have three facts: \eqref{eq-3.41}, $ x^k \to x^* $ from (1), and $ \left\langle \nabla f(x^k), d^k \right\rangle \le \rho \|d^k \|^2 $ from Lemma \ref{lem-2-1}, which together with \cite[Theorem 3.3]{facchinei1995minimization} allow us to claim that eventually the stepsize $ \alpha_k $ determined by the Armijo rule is 1, namely $ \alpha_k = 1 $. Then, the sequence converges quadratically, completing the proof.
\end{proof}

\section{Numerical results}
In this section, we report the numerical results of the modified block Newton algorithm (MBNL0R) proposed in this paper and compare it with NL0R  \cite{zhou2021newton}\footnote{The code is available via https://github.com/ShenglongZhou/NL0R}, PIHT \cite{lu2014iterative}, MIRL1 \cite{10.1093/imaiai/iaw002}\footnote{The code is available via https://github.com/ShenglongZhou/MIRL1}, PDASC \cite{JIAO2015400}\footnote{The code is available via http://www0.cs.ucl.ac.uk/staff/b.jin/companioncode.html}, HTP \cite{doi:10.1137/100806278}, CoSaMP \cite{NEEDELL2009301}. The code of the algorithm is implemented in MATLAB R2020a and computed on a laptop with AMD Ryzen 7 5800H with Radeon Graphics 3.20 GHz CPU and 16 GB memory. The parameter settings of MBNL0R is the same as  NL0R in \cite{zhou2021newton} except for $\mu_k$. In MBNL0R, we set the regularization parameter $\mu_k = \min (0.1, \|F_{\tau}(x^k, T_k)\|^2)$.

We will report the following results: dimension $n$, number of iterations, computation time (in seconds), and res. Here, res = $\|x^k - x^*\|$, where $x^*$ is the ground truth signal.

\subsection{Randomly Examples on Compressed sensing}
Compressing sensing has seen revolutionary advances both in theory and algorithm development over the past decade. In this subsection, we conducted a comparative analysis with PIHT \cite{lu2014iterative}, MIRL1 \cite{10.1093/imaiai/iaw002}, PDASC \cite{JIAO2015400}, NL0R  \cite{zhou2021newton}. We consider the exact recovery $y = Ax^*$ in \textbf{E1} and \textbf{E2} and the exact recovery $y = Ax^* + \xi$ in \textbf{E3} and \textbf{E4}, where $\xi$ is the noise and $A \in \mathbb{R}^{m\times n}$ will be described in example. The "ground truth" signal $x^*$ are iid samples of the standard normal distribution, and their locations are picked randomly. We set the number of non-zero components $s^* = 0.01n$. $f(x)$ takes the form of $f(x) = \frac{1}{2} \|Ax - y\|^2_2$. We run 20 trials for each example and take the average of the results and round the number of iterations iter to the nearest integer.

\textbf{E1} Let $A \in \mathbb{R}^{m\times n}$ are random Gaussian matrix. We consider the exact recovery $y = Ax^*$.

\textbf{E2} Let $A = BC \in \mathbb{R}^{n\times n}$ where $B \in \mathbb{R}^{n\times m}$, $C\in \mathbb{R}^{m\times n}$ are random Gaussian matrix. We consider the exact recovery $y = Ax^*$.

\textbf{E3} Let $A = BC \in \mathbb{R}^{n\times n}$ where $B \in \mathbb{R}^{n\times m}$, $C\in \mathbb{R}^{m\times n}$ are random Gaussian matrix. We consider the exact recovery $y = Ax^* + c\xi$, where $\xi$ is white Gaussian noise which code by \text{randn(m,1)} and noise level $c$ is set to 0.001.

\textbf{E4} Let $A \in \mathbb{R}^{m\times n}$ are random Gaussian matrix. We consider the exact recovery $y = Ax^* + c\xi$, where $\xi$ is white Gaussian noise which code by \text{randn(m,1)} and noise level $c$ is set to 0.001.

\begin{figure}[h]
	\centering
	\includegraphics[width=6cm]{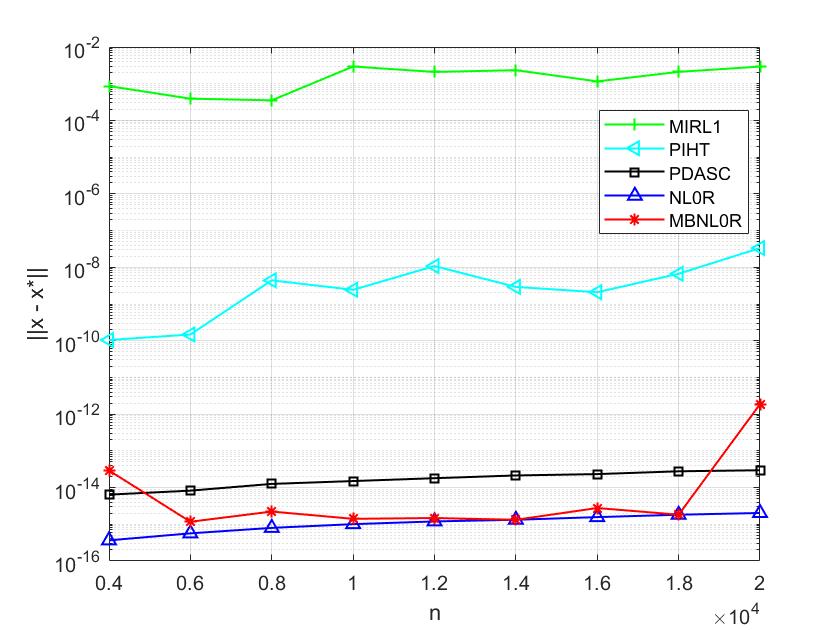
}
    \includegraphics[width=6cm]{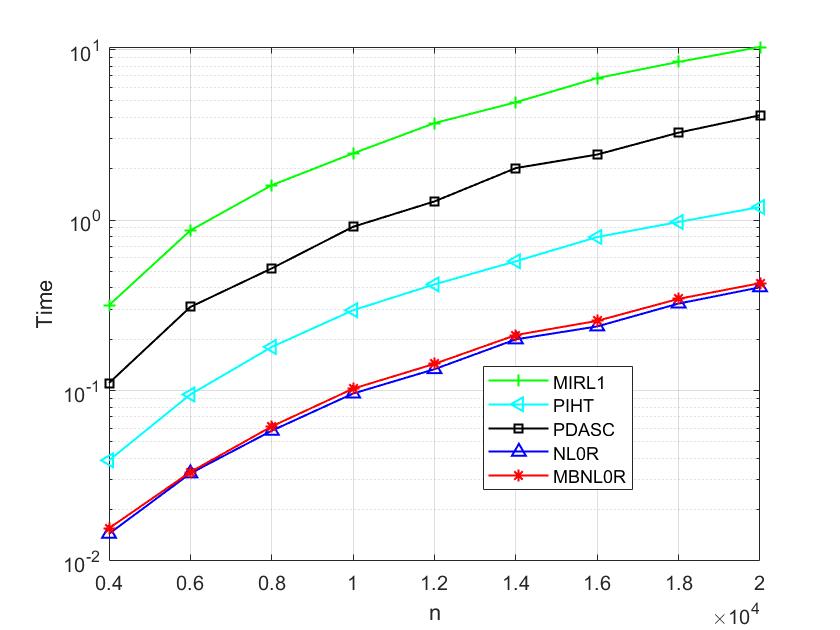
}
	\caption{Average recovery results for \textbf{E1} with $m = 0.25n$}
	\label{fig1}
\end{figure}

\begin{figure}[h]
	\centering
	\includegraphics[width=6cm]{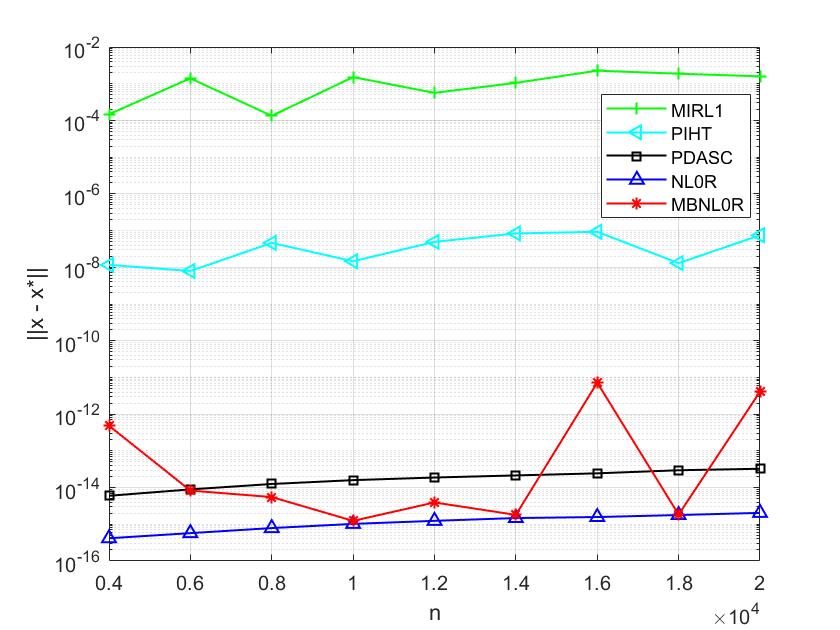
}
    \includegraphics[width=6cm]{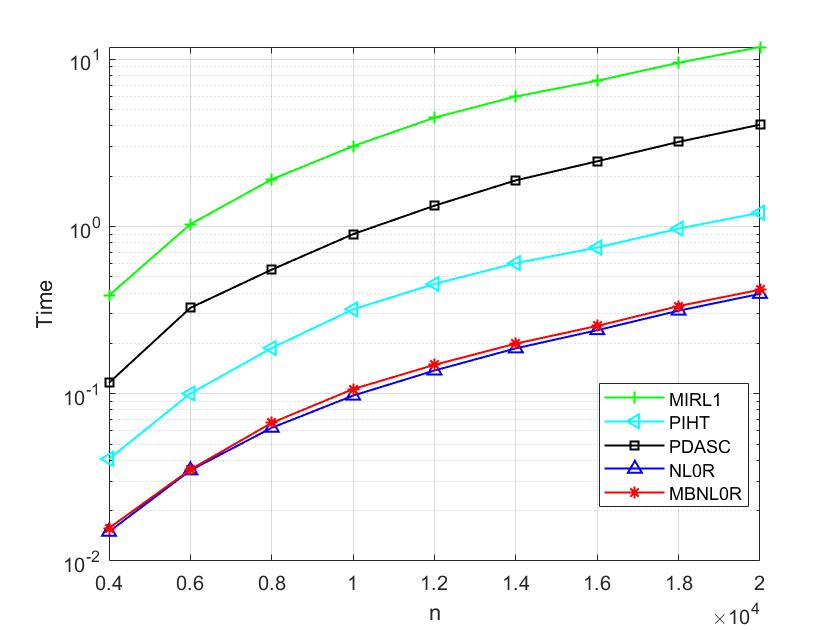
}
	\caption{Average recovery res and time for \textbf{E2} with $m = 0.25n$}
	\label{fig2}
\end{figure}

\begin{figure}[h]
	\centering
	\includegraphics[width=6cm]{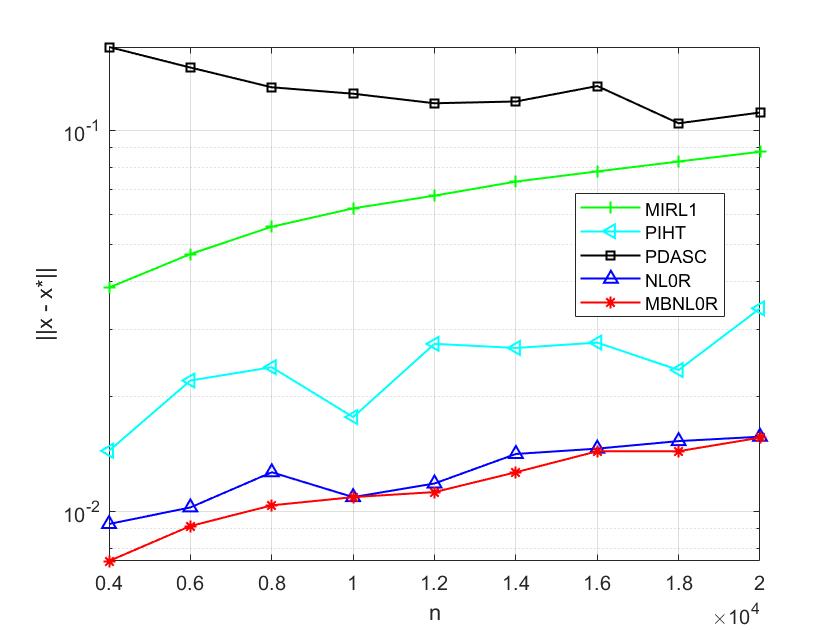
}
    \includegraphics[width=6cm]{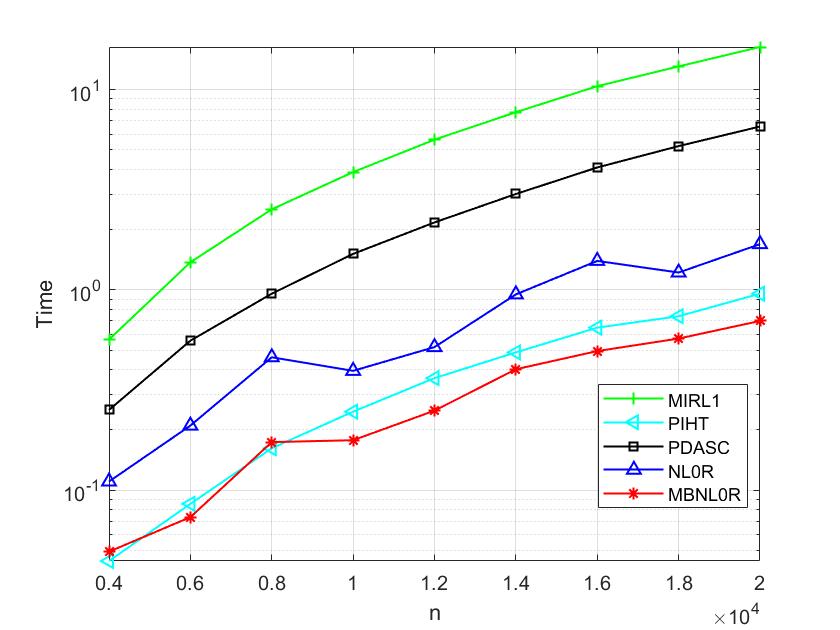
}
	\caption{Average recovery res and time for \textbf{E3}}
	\label{fig3}
\end{figure}

\begin{figure}[h]
	\centering
	\includegraphics[width=6cm]{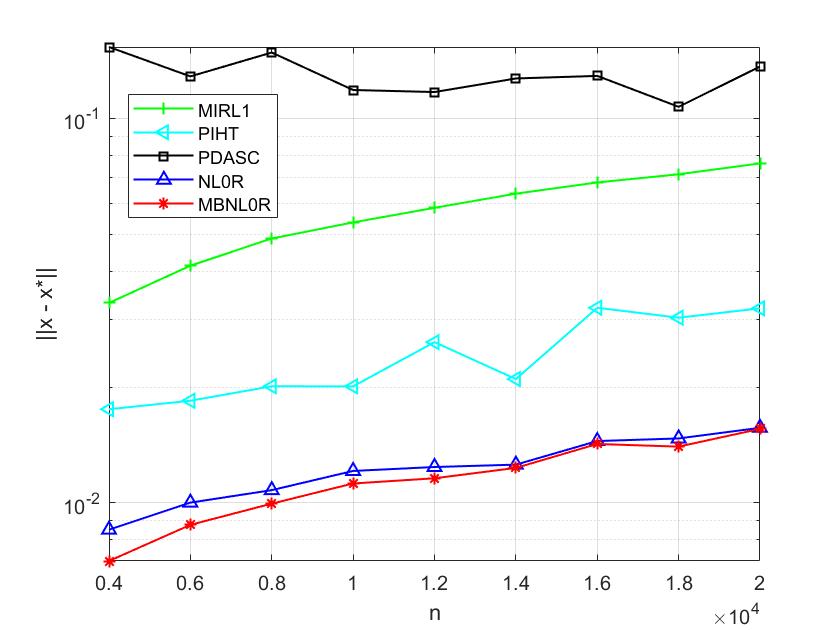
}
    \includegraphics[width=6cm]{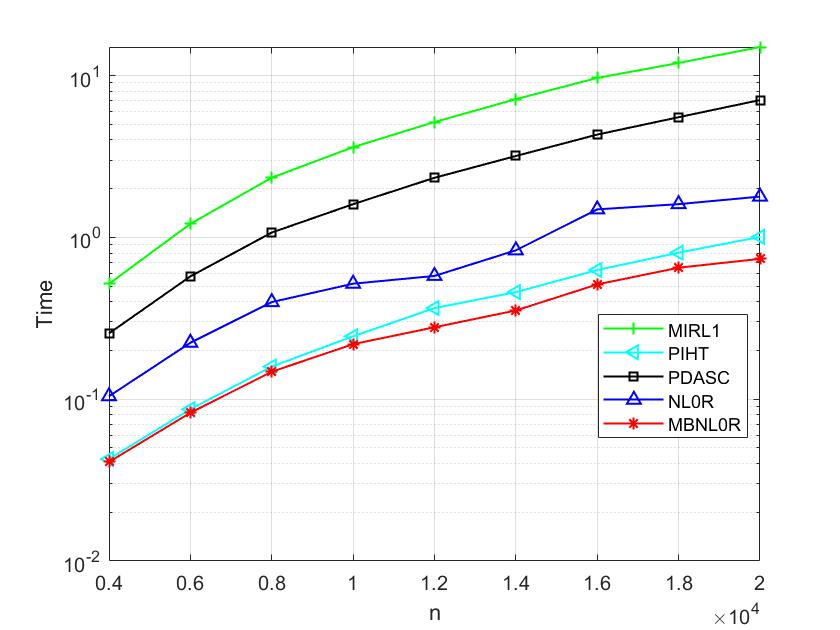
}
	\caption{Average recovery res and time for \textbf{E4} with $m = 0.25n$}
	\label{fig4}
\end{figure}

\begin{table*}[h]
\scriptsize
\centering
\caption{\text{Results on \textbf{E3} with $m = 0.25n$}}\label{Tab03}
\begin{tabular*}{\textwidth}{@{\extracolsep\fill}ccccrcccr}
\toprule
 %\multicolumn{7}{c}{\text{ExamplE1}}  \\
\textbf{Algorithm}   &{n} &  iter     &  time   &   res
&{n}  &  iter      &  time   &   res \\
 \midrule
\text{MIRL1} & 6000  & 02  & 1.37  & 4.73e-02 & 8000  & 02  & 2.53  & 5.58e-02 \\
\text{PIHT} & 6000  & 15  & 0.09  & 2.20e-02 & 8000  & 14  & 0.16  & 2.39e-02 \\
\text{PDASC} & 6000  & 51  & 0.21  & 1.46e-01 & 8000  & 51  & 0.96  & 1.30e-01 \\
\text{NL0R} & 6000  & 20  & 0.21  & 1.02e-02 & 8000  & 19  & 0.46  & 1.26e-02 \\
\text{MBNL0R} & 6000  & 18  & 0.07  & 9.14e-03 & 8000  & 21  & 0.17  & 1.04e-02 \\
 \midrule
\text{MIRL1} & 10000  & 02  & 3.88  & 6.23e-02 & 12000  & 02  & 5.63  & 6.73e-02 \\
\text{PIHT} & 10000  & 14  & 0.25  & 1.77e-02 & 12000  & 14  & 0.36  & 2.75e-02 \\
\text{PDASC} & 10000  & 52  & 0.39  & 1.25e-01 & 12000  & 52  & 2.17  & 1.18e-01 \\
\text{NL0R} & 10000  & 18  & 0.39  & 1.09e-02 & 12000  & 17  & 0.52  & 1.18e-02 \\
\text{MBNL0R} & 10000  & 16  & 0.18  & 1.09e-02 & 12000  & 16  & 0.25  & 1.12e-02 \\
 \midrule
\text{MIRL1} & 14000  & 02  & 7.72  & 7.33e-02 & 16000  & 02  & 10.41  & 7.79e-02 \\
\text{PIHT} & 14000  & 14  & 0.49  & 2.68e-02 & 16000  & 14  & 0.65  & 2.77e-02 \\
\text{PDASC} & 14000  & 52  & 0.95  & 1.19e-01 & 16000  & 52  & 4.09  & 1.31e-01 \\
\text{NL0R} & 14000  & 18  & 0.95  & 1.41e-02 & 16000  & 20  & 1.40  & 1.46e-02 \\
\text{MBNL0R} & 14000  & 18  & 0.40  & 1.26e-02 & 16000  & 17  & 0.49  & 1.44e-02 \\
 \midrule
\text{MIRL1} & 18000  & 02  & 13.06  & 8.28e-02 & 20000  & 02  & 16.33  & 8.78e-02 \\
\text{PIHT} & 18000  & 14  & 0.74  & 2.35e-02 & 20000  & 15  & 0.96  & 3.41e-02 \\
\text{PDASC} & 18000  & 53  & 1.22  & 1.04e-01 & 20000  & 52  & 6.54  & 1.11e-01 \\
\text{NL0R} & 18000  & 19  & 1.22  & 1.53e-02 & 20000  & 19  & 1.69  & 1.57e-02 \\
\text{MBNL0R} & 18000  & 17  & 0.57  & 1.44e-02 & 20000  & 17  & 0.70  & 1.56e-02 \\
\bottomrule
\end{tabular*}
\end{table*}

\begin{table*}[h]
\scriptsize
\centering
\caption{\text{Results on \textbf{E4} with $m = 0.25n$}}\label{Tab04}
\begin{tabular*}{\textwidth}{@{\extracolsep\fill}ccccrcccr}
\toprule
 %\multicolumn{7}{c}{\text{ExamplE1}}  \\
\textbf{Algorithm}   &{n} &  iter     &  time   &   res
&{n}  &  iter      &  time   &   res \\
\midrule
\text{MIRL1} & 6000  & 02  & 1.21  & 4.14e-02 & 8000  & 02  & 2.34  & 4.87e-02 \\
\text{PIHT} & 6000  & 14  & 0.09  & 1.84e-02 & 8000  & 14  & 0.16  & 2.01e-02 \\
\text{PDASC} & 6000  & 52  & 0.22  & 1.29e-01 & 8000  & 53  & 1.07  & 1.48e-01 \\
\text{NL0R} & 6000  & 19  & 0.22  & 1.00e-02 & 8000  & 20  & 0.40  & 1.08e-02 \\
\text{MBNL0R} & 6000  & 18  & 0.08  & 8.76e-03 & 8000  & 17  & 0.15  & 9.95e-03 \\
 \midrule
\text{MIRL1} & 10000  & 02  & 3.61  & 5.36e-02 & 12000  & 02  & 5.16  & 5.85e-02 \\
\text{PIHT} & 10000  & 14  & 0.24  & 2.01e-02 & 12000  & 14  & 0.37  & 2.62e-02 \\
\text{PDASC} & 10000  & 53  & 0.52  & 1.19e-01 & 12000  & 53  & 2.33  & 1.17e-01 \\
\text{NL0R} & 10000  & 20  & 0.52  & 1.21e-02 & 12000  & 18  & 0.58  & 1.24e-02 \\
\text{MBNL0R} & 10000  & 19  & 0.22  & 1.12e-02 & 12000  & 17  & 0.28  & 1.16e-02 \\
 \midrule
\text{MIRL1} & 14000  & 02  & 7.14  & 6.37e-02 & 16000  & 02  & 9.66  & 6.81e-02 \\
\text{PIHT} & 14000  & 14  & 0.46  & 2.10e-02 & 16000  & 14  & 0.63  & 3.22e-02 \\
\text{PDASC} & 14000  & 53  & 0.83  & 1.27e-01 & 16000  & 52  & 4.32  & 1.29e-01 \\
\text{NL0R} & 14000  & 19  & 0.83  & 1.26e-02 & 16000  & 20  & 1.49  & 1.45e-02 \\
\text{MBNL0R} & 14000  & 16  & 0.35  & 1.23e-02 & 16000  & 17  & 0.51  & 1.42e-02 \\
 \midrule
\text{MIRL1} & 18000  & 02  & 11.96  & 7.16e-02 & 20000  & 02  & 15.00  & 7.64e-02 \\
\text{PIHT} & 18000  & 15  & 0.80  & 3.03e-02 & 20000  & 15  & 1.01  & 3.20e-02 \\
\text{PDASC} & 18000  & 53  & 1.60  & 1.07e-01 & 20000  & 53  & 7.05  & 1.36e-01 \\
\text{NL0R} & 18000  & 19  & 1.60  & 1.47e-02 & 20000  & 20  & 1.79  & 1.57e-02 \\
\text{MBNL0R} & 18000  & 17  & 0.65  & 1.40e-02 & 20000  & 17  & 0.74  & 1.56e-02 \\
\bottomrule
\end{tabular*}
\end{table*}

\begin{figure}[h]
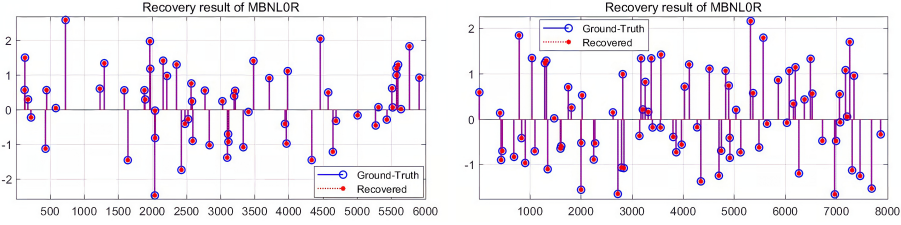

	\centering
	\includegraphics[width=6cm]{recovery_result_6000.jpg
}
    \includegraphics[width=6cm]{recovery_result_8000.jpg
}
	\caption{Recovery result for \textbf{E4} $m = 0.25n$}
	\label{fig8}
\end{figure}
The results for \textbf{E1} to \textbf{E4} are reported in Table \ref{Tab03} - \ref{Tab04} and Fig. \ref{fig1} - Fig. \ref{fig4}. From Fig. \ref{fig1} and Fig. \ref{fig2}, it shows that the performance of our proposed method is close to NL0R. In \textbf{E1}, both MBNL0R and NL0R significantly outperform CoSaMP, HTP, and PDASC in terms of computational efficiency and accuracy. In \textbf{E2}, NL0R shows to be the most efficient method. Although MBNL0R exhibits slightly lower accuracy than NL0R, its computational time remains comparable to that of NL0R. As can be observed from Figs. \ref{fig3} and \ref{fig4}, MBNL0R demonstrates a more pronounced advantage in \textbf{E3} and \textbf{E4}.
It shows that, both MBNL0R and NL0R maintain superior computational accuracy among contemporary algorithms, with MBNL0R demonstrating a modest improvement for NL0R. As evidenced in Fig. \ref{fig3} and \ref{fig4}, MBNL0R achieves the most significant reduction in computational time among all compared algorithms, demonstrating particularly remarkable improvement over NL0R. Fig. \ref{fig8} demonstrates the reconstruction performance of MBNL0R on \textbf{E4}, with the left panel displaying results for $n=6000$ and the right panel for $n=8000$. As evident from Fig. \ref{fig8}, the solution obtained by MBNL0R exhibits excellent agreement with the ground truth signal $x^*$.
These further highlight the robustness and effectiveness of MBNL0R in noisy signal recovery scenarios.

\subsection{2-D Image Examples}
\textbf{E5} (2-D image data \cite{zhou2021newton})
Some images are naturally not sparse themselves but can be sparse under some wavelet transforms. Here, we take advantage of the Daubechies wavelet 1, denoted as $W(\cdot)$. Then the images under this transform (i.e., $x^* := W(\omega)$) are sparse, and $\omega$ is the vectorized intensity of an input image. Therefore, the explicit form of the sampling matrix may not be available. We consider the exact recovery $y = Ax^* + \xi$. $f(x)$ takes the form of $f(x) = \frac{1}{2} \|Ax - y\|^2_2$. We consider a sampling matrix taking the form $A = FW^{-1}$, where $F$ is the partial fast Fourier transform, and $W^{-1}$ is the inverse of $W$. Finally, the added noise $\xi$ has each element $\xi_i \sim n_f \cdot N$ where $N$ is the standard normal distribution and $n_f$ is the noise factor. Three typical choices of $n_f$ are considered, namely $n_f \in \{0.01, 0.1\}$. For this experiment, we compute a gray image (see the original image in Fig. \ref{fig5}) with size $512 \times 512$ (i.e., $n = 512^2 = 262,144$) and the sampling sizes $m = 20,033$. We compute the peak signal to noise ratio (PSNR) defined by $\text{PSNR} := 10\log_{10}(n\|x - x^*\|^{-2})$ to measure the performance of the method. Note that a larger PSNR implies a better performance.

\begin{figure}[htbp]
	\centering
	\includegraphics[width=10cm]{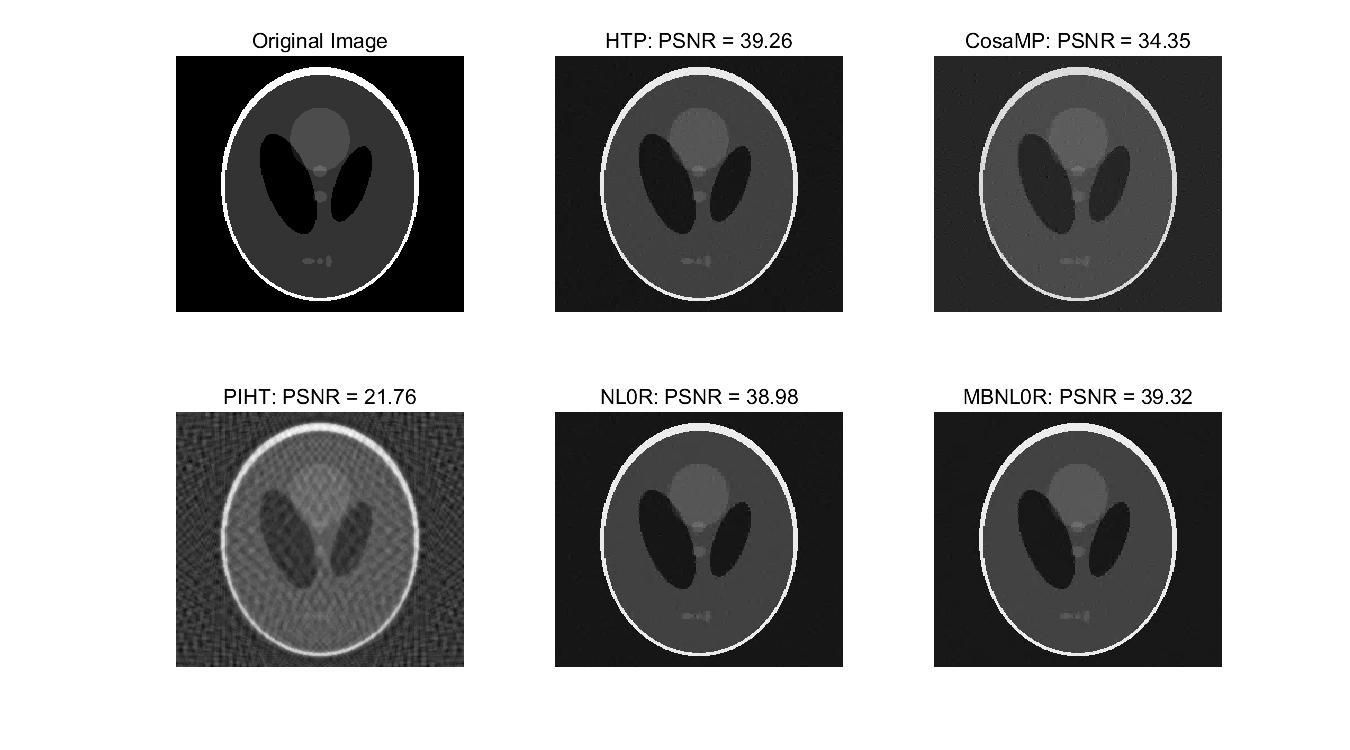
}
	\caption{Recovery results for \textbf{E5} with $nf = 0.01$}
	\label{fig5}
\end{figure}

\begin{figure}[htbp]
	\centering
	\includegraphics[width=10cm]{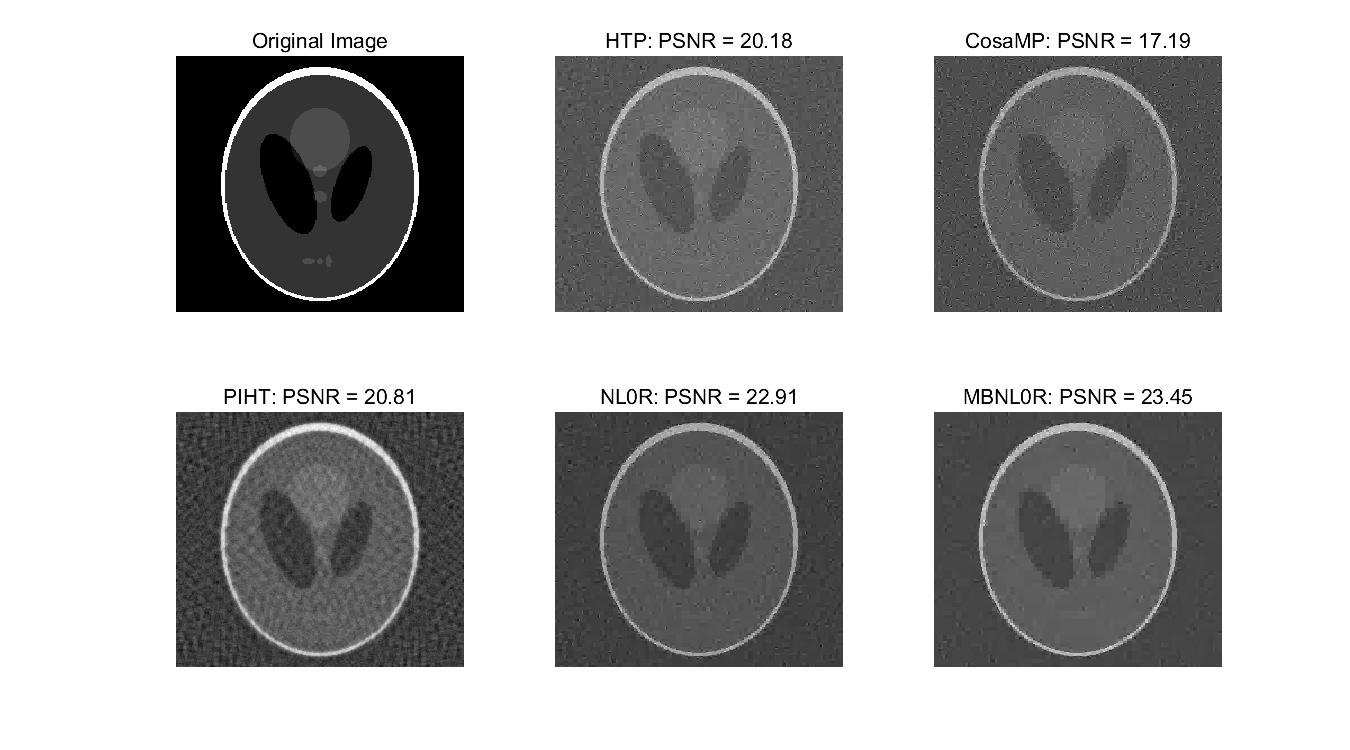
}
	\caption{Recovery results for \textbf{E5} with $nf = 0.1$}
	\label{fig6}
\end{figure}

\begin{table*}[htbp]
\scriptsize
\centering
\caption{\text{Results on 2-D image examples}}\label{Tab05}
\begin{tabular*}{\textwidth}{@{\extracolsep\fill}ccccccc}
\toprule
\  & \multicolumn{3}{c}{0.01} & \multicolumn{3}{c}{0.1} \\
\cmidrule(r){2-4} \cmidrule(r){5-7}
{Algorithm} &  PSNR      &  time   &   $\|x\|_0$

&  PSNR      &  time   &   $\|x\|_0$ \\
\midrule
HTP     & 39.26  & 2.55  & 4000  & 20.18 & 2.83  & 4000   \\

CoSaMP  & 34.35  & 5.48  & 4000  & 17.19  & 5.54  & 4000   \\

PIHT    & 21.76  & 1.23e-1  & 28124  & 20.81  & 8.18e-2  & 22298
\\

NL0R     & 38.98  & 7.80e-1  & 3939  & 22.91 & 1.80e-1  & 3281   \\

MBNL0R     & 38.32  & 6.87e-1  & 3414  & 23.45 & 1.06e-1  & 1954   \\
\bottomrule
\end{tabular*}
\end{table*}
In this example, we select four methods CoSaMP \cite{NEEDELL2009301}, HTP \cite{doi:10.1137/100806278}, PIHT \cite{lu2014iterative} and NL0R \cite{zhou2021newton}. We first run NL0R for noise factor $n_f = 0.01$ that can deliver a solution with good sparsity level which approximately 4000. Then, we set the sparsity level $s = 4000$ for CoSaMP and HTP since they need such prior information. The results for \textbf{E5} are presented in Table \ref{Tab05} and Fig. \ref{fig5} - Fig. \ref{fig6}. From Table \ref{Tab05}, it shows that our proposed method has a slight advantage in terms of time and PSNR for \textbf{E5} with noise factor $n_f = 0.1$.

\subsection{Sparse linear complementarity problem}
Sparse linear complementarity problems have been applied in real-world applications such as bimatrix games and portfolio selection problems. These problems aim to find a sparse vector $x\in \mathbb{R}^{n}$ from $\Omega := \{ x\in \mathbb{R}^{n}\ : x\ge 0,\ Mx + q \ge 0,\ \langle x, Mx + q \rangle = 0 \}$, where $M \in \mathbb{R}^{n\times n}$ and $q\in \mathbb{R}^{n}$. A point $x\in \Omega$ is equivalent to
\begin{align}
f(x) := \sum_{i=1}^{n}\phi(x_i, M_i x + q_i) = 0
\end{align}
where $M_i$ is the $i$th row of $M$ and $\phi$ is the so-called NCP function that is defined by $\phi(a,b) = 0$ if and only if $a \ge 0,\ b\ge 0,\ ab = 0$. We choose an NCP function $\phi(a,b) = a_{+}^2 b_{+}^2 + (-a)_{+}^2 + (-b)_{+}^2$ which is the same as in \cite{zhou2021newton}.

\textbf{E6} \cite{zhou2021newton} Let $M = ZZ^{\top}$ with $Z \in \mathbb{R}^{n \times m}$ and $m \leq n$ (e.g., $m = n/2$). Elements of $Z$ are i.i.d. samples from the standard normal distribution. Each column is then normalized to have unit length and $q$ is obtained by
\begin{align*}
q_i =
\begin{cases}
-(Mx^*)_i & \text{if } x^*_i > 0 \\
|(Mx^*)_i| & \text{otherwise}
\end{cases}
\end{align*}
We run 20 trials for each example and take the average of the results and round the number of iterations iter to the nearest integer.

\begin{figure}[htbp]
	\centering
	\includegraphics[width=6cm]{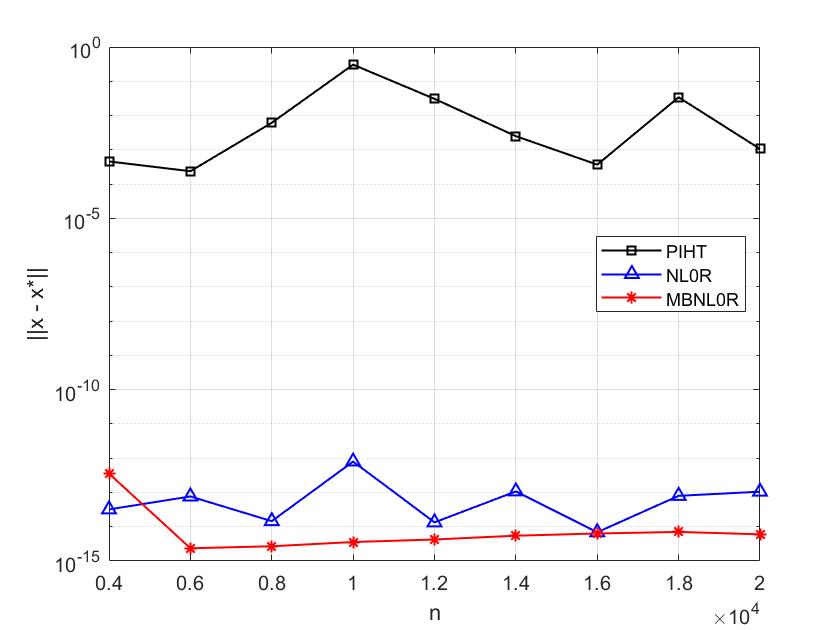
}
	\includegraphics[width=6cm]{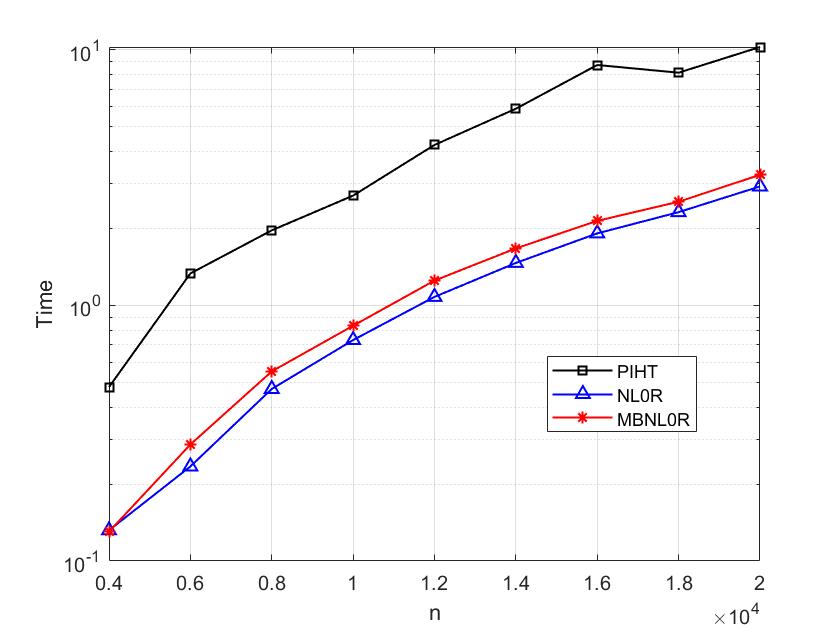
}
	\caption{Average recovery error and time for \textbf{E6}}
	\label{fig7}
\end{figure}

\begin{table*}[htbp]
\scriptsize
\centering
\caption{\text{Results on \textbf{E6}}}\label{Tab06}
\begin{tabular*}{\textwidth}{@{\extracolsep\fill}ccccrcccr}
\toprule
 %\multicolumn{7}{c}{\text{ExamplE1}}  \\
\textbf{Algorithm}   &{n} &  iter     &  time   &   res
&{n}  &  iter      &  time   &   res \\
 \midrule
\text{PIHT} & 6000  & 23  & 1.33  & 2.37e-04 & 8000  & 19  & 1.97  & 6.16e-03 \\
\text{NL0R} & 6000  & 08  & 0.23  & 7.65e-14 & 8000  & 09  & 0.47  & 1.45e-14 \\
\text{MBNL0R} & 6000  & 09  & 0.29  & 2.36e-15 & 8000  & 10  & 0.55  & 2.69e-15 \\
 \midrule
\text{PIHT} & 10000  & 16  & 2.69  & 3.05e-01 & 12000  & 17  & 4.24  & 3.07e-02 \\
\text{NL0R} & 10000  & 09  & 0.73  & 7.98e-13 & 12000  & 09  & 1.08  & 1.34e-14 \\
\text{MBNL0R} & 10000  & 10  & 0.83  & 3.58e-15 & 12000  & 10  & 1.25  & 4.24e-15 \\
 \midrule
\text{PIHT} & 14000  & 19  & 5.88  & 2.50e-03 & 16000  & 23  & 8.71  & 3.69e-04 \\
\text{NL0R} & 14000  & 09  & 1.47  & 1.07e-13 & 16000  & 09  & 1.91  & 6.91e-15 \\
\text{MBNL0R} & 14000  & 10  & 1.67  & 5.48e-15 & 16000  & 10  & 2.14  & 6.36e-15 \\
 \midrule
\text{PIHT} & 18000  & 16  & 8.13  & 3.41e-02 & 20000  & 18  & 10.25  & 1.06e-03 \\
\text{NL0R} & 18000  & 09  & 2.32  & 7.98e-14 & 20000  & 09  & 2.92  & 1.05e-13 \\
\text{MBNL0R} & 18000  & 10  & 2.54  & 7.05e-15 & 20000  & 10  & 3.24  & 6.01e-15 \\
\bottomrule
\end{tabular*}
\end{table*}
In this example, we select two methods: PIHT \cite{lu2014iterative} and NL0R \cite{zhou2021newton}.
The results for \textbf{E6} is reported in Table \ref{Tab06} and Fig. \ref{fig7}. We vary the sample size $n$ but fix $m = n/2$, and set the number of non-zero components $s^* = 0.01n$. From Table \ref{Tab06} and Fig. \ref{fig7}, it shows that, both NL0R and MBNL0R exhibit substantially higher computational efficiency than PIHT. Particularly, MBNL0R achieves a little advantages over NL0R in terms of computational accuracy and time efficiency.

\section{Conclusion}
In this paper, we proposed a globally convergent modified block Newton method(MBNL0R) for $\ell_0$ regularized optimization. The main difference between MBNL0R and NL0R \cite{zhou2021newton} lies in the use of approximate Jacobian matrix and the regularization. We prove that the proposed method is globally convergent. Leveraging the special structural properties of the NL0R method, we established the quadratic convergence rate for MBNL0R. The numerical results reveal the efficiency of our proposed method.

\section*{Funding}
This work is supported by the National Natural Science Foundation of China  [grant number 12071032], [grant number 12271526].

\bibliography{sn-bibliography}%

\end{document}